\documentclass[reqno]{amsart}
\usepackage[T1]{fontenc}
\usepackage[latin1]{inputenc}
\usepackage{amssymb}
\usepackage{amsfonts}
\usepackage{amsmath,mathtools}
\usepackage{mathrsfs}
\usepackage{graphicx}
\usepackage{color}
\usepackage{hyperref}
\usepackage{esint}
\usepackage{verbatim}
\usepackage{bbm}
\usepackage{tikz}
\usepackage{setspace}
\usepackage{stackrel}

\definecolor{vg}{rgb}{0.0, 0.4, 0.1}

\newcommand{\N}[1]{\mathbb{N}^{#1}}
\newcommand{\R}[1]{\mathbb{R}^{#1}}
\renewcommand{\S}[1]{\mathbb{S}^{#1}}

\newcommand{\cC}{\mathcal C}

\newcommand{\cH}{\mathcal H}

\newcommand{\cL}{\mathcal L}
\newcommand{\cM}{\mathcal M}

\newcommand{\bulk}{\mathrm{bulk}}

\newcommand{\de}{\mathrm d}

\newcommand{\surface}{\mathrm{surf}}

\newcommand{\eps}{\varepsilon}

\newcommand{\norma}[1]{\left\lVert#1\right\rVert}

\newcommand{\tr}{\operatorname{tr}}

\newcommand{\wto}{\rightharpoonup}
\newcommand{\lwto}{\longrightharpoonup}
\newcommand{\wsto}{\stackrel{*}{\rightharpoonup}}

\newcommand{\wSD}[1]{\stackrel[SD]{#1}\lwto}
\newcommand{\Hto}{\stackbin[H]{}{\to}}

\renewcommand{\geq}{\geqslant}
\renewcommand{\leq}{\leqslant}

\newcommand{\longrightharpoonup}{\relbar\joinrel\rightharpoonup}

\newcommand{\weakst}{\stackrel{*}{\longrightharpoonup}}
\newcommand{\pweak}{\stackrel{p}{\longrightharpoonup}}

\newcommand{\average}{{\mathchoice {\kern1ex\vcenter{\hrule
				height.4pt width 8pt depth0pt}
			\kern-11pt} {\kern1ex\vcenter{\hrule height.4pt width 4.3pt
				depth0pt} \kern-7pt} {} {} }}
\newcommand{\ave}{\average\int}

\newcommand{\res}{\mathop{\hbox{\vrule height 7pt width .5pt depth
			0pt\vrule height .5pt width 6pt depth0pt}}\nolimits}

\mathchardef\emptyset="001F

\makeatletter
 \@namedef{subjclassname@2020}{%
 \textup{2020} Mathematics Subject Classification}
 \makeatother

\providecommand{\U}[1]{\protect\rule{.1in}{.1in}}
\numberwithin{equation}{section}
\newtheorem{definition}{Definition}[section]
\newtheorem{theorem}[definition]{Theorem}
\newtheorem{lemma}[definition]{Lemma}

\theoremstyle{definition} {\newtheorem{remark}[definition]{Remark}}

\makeindex

\title[A Global Method for Relaxation for Multi-levelled Structured Deformations
]{A Global Method for Relaxation for Multi-levelled Structured Deformations}

\author[A.~C.~Barroso]{Ana Cristina Barroso}
\address[A.~C.~Barroso]{Departamento de Matem\'atica and CMAFcIO, 
	Faculdade de Ci\^encias da Universidade de Lisboa,
	Campo Grande, Edif\' \i cio C6, Piso 1,
	1749-016 Lisboa, Portugal}
\email{acbarroso@ciencias.ulisboa.pt}
\author[J.~Matias]{Jos\'{e} Matias}
\address[J.~Matias]{Departamento de Matem\'atica, Instituto Superior T\'ecnico, Av.~Rovisco Pais, 1, 1049-001 Lisboa, Portugal}
\email{jose.c.matias@tecnico.ulisboa.pt}

\author[E.~Zappale]{Elvira Zappale}
\address[E.~Zappale]{Dipartimento di Scienze di Base ed applicate per l'Ingegneria, Sapienza - Universit\`{a} di Roma, Via A. Scarpa, 16, 00161, Roma, Italy}
\email{elvira.zappale@uniroma1.it}

\date{\today}

\subjclass[2020]
{49J45, 46E30, 74A60, 74M99, 74B20.}
	
\keywords{global method for relaxation, hierarchical system of structured deformations, multiscale geometry, elasticity, disarrangements, integral representation}

\makeindex

\begin{document}

\begin{abstract}
We prove an integral representation result for a class of variational functionals appearing in the framework of hierarchical systems of structured deformations via a global method for relaxation.  
Some applications to specific relaxation problems are also provided.
\end{abstract}

\maketitle

\tableofcontents

\section{Introduction}

Our purpose in this paper is to establish a global method for relaxation
applicable in the context of multi-levelled structured deformations. The aim is to
provide an integral representation for a class of functionals, defined in the set of $(L+1)$-level (first-order) structured deformations (see Definition \ref{Def2.1}), via the study of a related local Dirichlet-type problem, and to identify the corresponding relaxed energy densities, under quite general assumptions (we refer to \cite{BFM, BFLM} for the introduction of the method in the $BV$ and $SBV_p$ contexts). 

First-order structured deformations were introduced by Del Piero and Owen \cite{DPO1993} in order to provide a mathematical framework that captures the effects at the macroscopic level of smooth deformations and of non\--\-smooth deformations (disarrangements) at one sub-macroscopic level. 
In the classical theory of mechanics, the deformation of the body is characterised exclusively by the macroscopic deformation field~$g$ and its gradient~$\nabla g$. 
In the framework of structured deformations, an additional geometrical field~$G$ is introduced with the intent to capture the contribution at the macroscopic scale of smooth sub-macroscopic changes, while the difference $\nabla g-G$ captures the contribution at the macroscopic scale of non-smooth sub-macroscopic changes, such as slips and separations (referred to as \emph{disarrangements} \cite{DO2002}). The field~$G$ is called the deformation without disarrangements, and, heuristically, the disarrangement tensor $M\coloneqq \nabla g-G$ is an indication of how ``non-classical'' a structured deformation is.
This broad theory is rich enough to address mechanical phenomena such as elasticity, plasticity and the behaviour of crystals with defects.

The variational formulation for first-order structured deformations in the $SBV$ setting was first addressed by Choksi and Fonseca \cite{CF1997} where a (first-order) structured deformation is defined to be a pair 
$(g,G)\in SBV(\Omega;\R{d})\times L^p(\Omega;\R{d\times N}), \; p \geq 1.$
Departing from 
a functional which associates to any deformation $u$ of the body an energy featuring a bulk contribution, which measures the deformation (gradient) throughout the whole body, and an interfacial contribution, accounting for the energy needed to fracture the body, an integral representation for the ``most economical way'' to approach a given structured deformation was derived. 

The theory of first-order structured deformations was broadened by Owen and Paroni \cite{OP2000} to second\--\-order structured deformations, which also account for other geometrical changes, such as curvature, at one sub-macroscopic level. The variational formulation in the $SBV^2$ setting for second-order structured deformations was carried out by Barroso, Matias, Morandotti and Owen \cite{BMMO2017}. This last formulation allows for jumps on both the approximating fields, as well as on their gradients. In \cite{MMO}, and the references therein, the interested reader can find a comprehensive survey about the theory of structured deformations, as well as applications.

In the recent contribution \cite{DO2019}, Deseri and Owen took a further step, extending the theory of \cite{DPO1993} and the field theory in \cite{DO2003},
to hierarchical systems of structured
deformations in order to include the effects of disarrangements at more than one sub-macroscopic level.
Indeed, in the context of dynamics and large isothermal deformations, the field theory in \cite{DO2003} is a first step in a programme to employ structured deformations of continua in order to study the evolution of bodies that undergo smooth deformations at a macroscopic level, but can also dissipate energy at a sub-macroscopic level.
However, many natural and man-made materials, for example, muscles, cartilage, bones, plants and some biomedical materials, exhibit different levels of
disarrangements.  Moreover, toughening mechanisms, characterized by distributions
of sub-macroscopic separations at the various sub-macroscopic levels, do closely
follow the internal arrangements of such materials. For these reasons, Deseri and Owen extended the work in \cite{DO2003}, by proposing the generalized field theory in \cite{DO2019} which sharpens the description of the physical nature of dissipative mechanisms that can arise, including the effects of different levels of microstructure.

The main aim of our work is to provide a mathematical framework to address more general problems in the context of the richer theory in \cite{DO2019}. In particular, this work provides a tool allowing for the extension, to different sub-macroscopic levels, of the applications of the theory of structured deformations collected in \cite{MMO}.

In the setting of hierarchical systems of structured deformations, a first-order structured deformation $(g,G)$ corresponds to a two-level hierarchical system (the macroscopic level, $g$, plus one microscopic level, $G$), while, for $L>1$, an 
$L+1$-level hierarchical system consists of $(g, G_1, \ldots, G_{L})$, where each $G_i, i = 1, \ldots L$,  provides the effects at the macro-level of the corresponding sub-macroscopic level $i$.

A first approach to the mathematical formulation of hierarchical systems of structured deformations was considered in \cite{BMMOZ2022}, where the authors provide an approximation theorem for an $(L + 1)$-level structured deformation 
$(g,G_1, \ldots G_L)$ and propose the assignment of an energy to this
multi-levelled structured deformation by means of a well-posed recursive
relaxation process. This consists of iterated applications of the integral representation result for first-order structured deformations in \cite{CF1997}, considering at each step $k+1$ new energy densities given by \cite[Theorems 2.16 and 2.17]{CF1997} assuming as parameters the matrices in the position of the fields $G_i, i=1,\dots, k$, see \cite[Theorem 3.4]{BMMOZ2022}.

The global method for relaxation, introduced by Bouchitt\'e, Fonseca and Mascarenhas \cite{BFM} in the $BV$ setting, and later addressed in the $SBV_p$ setting by Bouchitt\'e, Fonseca, Leoni and Mascarenhas \cite{BFLM}, provides a general method for the identification of the integral representation of a class of
functionals defined on $BV(\Omega; \mathbb{R}^d)\times {\mathcal O}(\Omega)$, where ${\mathcal O}(\Omega)$ represents the family of open sets contained in $\Omega$. 
Since its inception, this global method for relaxation has known numerous applications and generalizations, in particular, very recently it was used in the context of variable exponent spaces, see \cite{SSS}, spaces of bounded deformations, see \cite{CFVG} and \cite{BMZ}, and  second-order structured deformations in the space $BH(\Omega; \mathbb{R}^d)\times {\mathcal O}(\Omega)$ by Fonseca, Hagerty and Paroni \cite{FHP}. Note that in the $BH$ case, only jumps on gradients are allowed.
 
In this work we obtain a global method for relaxation appropriate for the study of functionals defined on the space of $L +1$-levelled hierarchical systems of first-order structured deformations of any order $L \geq 1$.
Our method is general and covers, with the same proof, any choice of $L \in \mathbb N$, with no need of iterating procedures as a standard relaxation approach would demand (see \cite{BMMOZ2024} and the previous iterative method proposed in \cite{BMMOZ2022}).
It is also worthwhile to stress that, although the functional spaces used to model structured deformations are $BV$ and $L^p$, due to the nature of structured deformations (cf. Definition \ref{S000}), neither the global method for relaxation in \cite{BFM}, nor the one in \cite{BFLM}, can be simply extended in a naive way to consider product topologies. In particular, the considered convergences are not the ones used in the classical $SBV_p$ setting, hence Ambrosio's compactness theorem (see \cite[Theorems 4.7 and 4.8]{AFP}) cannot be applied, due to the lack of a uniform control on the length of the singular set of the approximant bounded energy sequences (cf. the coercivity condition \eqref{(psi1)}). This is, in fact, the reason why in the structured deformation setting there is a distinction between the part of the deformation arising as the limit  of the entire approximants and the part emerging as the limit of only their smooth parts (see the introductory comments in \cite{CF1997} for further details).

We also provide several applications, with the aim of showing that our main integral representation theorem, Theorem \ref{GMthmHSD}, already covers, and improves, in a unified way, several results available in the literature in the first-order structured deformation context.

As a first application of our general theorem, we are able to extend the integral representation for first-order structured deformations proved in \cite[Theorems 2.16 and 2.17]{CF1997}, and later generalized in \cite[Theorem 5.1]{MMOZ} to allow for an explicit dependence on the spatial variable,
to the case of Carath\'{e}odory bulk energy densities. This latter setting is more realistic, for instance, it allows for the modelling of materials with a very different behaviour from one point to another, as multigrain-type materials, or other types of mixtures, which also appear in the optimal design context (cf. for instance \cite{MMZ}). On the other hand, the assumptions on the surface energy density can be weakened, compared with \cite{CF1997}, although, in the inhomogeneous setting, the continuity with respect to the spatial variable is still needed due to the fact that, in this case, it is meaningless to consider Lebesgue measurability for elliptic integrands defined on $N-1$ dimensional sets. See Theorem \ref{representation} for the precise statement. In particular, under the same set of hypotheses considered therein, we recover the formulae in \cite{CF1997} (see also
\cite{BMMO2017} and \cite{MMOZ} for the inhomogeneous setting). We further show that, in the case $p>1$, the cell formulae for the interfacial relaxed energy density in  \cite{CF1997} and \cite{MMOZ}, still hold when the bulk energy densities are Carath\'{e}odory (see Theorems \ref{representation} and \ref{contrepresentation}).

We stress that a standard relaxation approach, mimicking the arguments in \cite{CF1997}, could also be achieved in the measurable bulk energy setting, but this would require the proof of many auxiliary steps, while a global method approach is more direct.

We point out that each step of the recursive relaxation procedure for multi-levelled hierarchies of structured deformations, presented in \cite{BMMOZ2022}, whose densities, at each stage, satisfy hypotheses \eqref{(W1)_p}-\eqref{(psi4)}, also fits into the scope of Theorem \ref{GMthmHSD}. We refer to Subsection \ref{MSD} for 
more details and also for a different energetic formulation relying on Theorem \ref{GMthmHSD} in its full generality.

Another natural application of our abstract result is to homogenization problems, such as the one conside\-red in \cite{AMMZ}. In that paper, only the case $p>1$ was treated under uniform continuity hypotheses on the densities, while here we can extend the result to include also $p=1$, Carathéodory bulk densities and weaker assumptions on the elliptic integrands.
In this setting, we depart from an energy of the form
$$E_\varepsilon(u)\coloneqq \int_\Omega W(x/\varepsilon,\nabla u(x))\,\de x
+\int_{\Omega\cap S_u} \psi(x/\varepsilon,[u](x),\nu_u(x))\,\de\cH^{N-1}(x),$$
where $\varepsilon \to 0^+$. Besides a periodicity condition, the densities $W$ and $\psi$ satisfy other hypotheses (cf. Theorem \ref{homogenization}) which ensure that the relaxed functional \eqref{102hom} can be placed in the setting of our main theorem.

As a further application, we recover the integral representation for one of the relaxed energies in \cite{BMMO2017}. In this paper, an integral representation for the relaxation of an energy arising in the context of second-order structured deformations is obtained. A simple argument allows for the decomposition of the relaxed functional $I$ as the sum of two terms, $I = I_1 + I_2$. Although $I_1$ does not fit the scope of our global method result, due to the topology which is considered in its definition, we will show that Theorem \ref{GMthmHSD} applies to $I_2$ (in order to avoid conflicting notation to that introduced in \eqref{102}, in Section \ref{appl} $I_2$ will be denoted by $J$). As in some of the previously mentioned applications, our global method for relaxation applies under milder assumptions than those considered in \cite[Theorem 5.7]{BMMO2017},
recovering for our densities the same expressions that were deduced in \cite{BMMO2017} when those hypotheses are considered.  

A classical approach to a relaxation result for hierarchical systems of first-order structured deformations with an arbitrary number of levels $L$, and the comparison both with the method implied in \cite{BMMOZ2022} and with this abstract formulation, will be the subject of a forthcoming work \cite{BMMOZ2024}. We emphasize that the method in \cite{BMMOZ2024}, although it is expected to provide more explicit formulae, requires a direct proof for each choice of the level $L \in \mathbb N$ in a iterated way, while the global method does not require any iterative process as outlined in Subsection \ref{MSD}.

The paper is organized as follows: in Section \ref{prelim} we set the notation which will be used in the sequel and recall the notion of a multi-levelled structured deformation, as well as the approximation theorem for these deformations. In Section \ref{gm} we state and prove our main theorem (see Theorem \ref{GMthmHSD}), whereas
Section \ref{appl} is devoted to the aforementioned applications of our abstract result.

\section{Preliminaries}\label{prelim}

\subsection{Notation} 

We will use the following notations
\begin{itemize}
	\item $\mathbb N$ denotes the set of natural numbers without the zero element;
	\item $\Omega \subset \mathbb R^{N}$ is a bounded, connected open set with Lipschitz boundary; 
	\item $\mathbb S^{N-1}$ denotes the unit sphere in $\mathbb R^N$;
	\item $Q\coloneqq (-\tfrac12,\tfrac12)^N$ denotes the open unit cube of $\mathbb R^{N}$ centred at the origin; for any $\nu\in\mathbb S^{N-1}$, $Q_\nu$ denotes any open unit cube in $\mathbb R^{N}$ with two faces orthogonal to $\nu$; 
	\item for any $x\in\mathbb R^{N}$ and $\delta>0$, $Q(x,\delta)\coloneqq x+\delta Q$ denotes the open cube in $\mathbb R^{N}$ centred at $x$ with side length $\delta$; likewise $Q_\nu(x,\delta)\coloneqq x+\delta Q_\nu$;
	\item ${\mathcal O}(\Omega)$ is the family of all open subsets of $\Omega $, whereas ${\mathcal O}_\infty(\Omega)$ is the family of all open subsets of $\Omega $ with Lipschitz boundary; 
	\item $\mathcal L^{N}$ and $\mathcal H^{N-1}$ denote the  $N$-dimensional Lebesgue measure and the $\left(  N-1\right)$-dimensional Hausdorff measure in $\mathbb R^N$, respectively; the symbol $\de x$ will also be used to denote integration with respect to $\mathcal L^{N}$; 
	\item $\mathcal M(\Omega;\mathbb R^{d\times N})$ is the set of finite matrix-valued Radon measures on $\Omega$; $\mathcal M ^+(\Omega)$ is the set of non-negative finite Radon measures on $\Omega$;
	given $\mu\in\mathcal M(\Omega;\mathbb R^{d\times N})$,  
	the measure $|\mu|\in\mathcal M^+(\Omega)$ 
	denotes the total variation of $\mu$;
	\item $SBV(\Omega;\mathbb R^d)$ is the set of vector-valued \emph{special functions of bounded variation} defined on $\Omega$. 
	Given $u\in SBV(\Omega;\mathbb R^d)$, its distributional gradient $Du$ admits the decomposition $Du=D^au+D^su=\nabla u\cL^N+[u]\otimes\nu_u\mathcal H^{N-1}\res S_u$, where $S_u$ is the jump set of~$u$, $[u]$ denotes the jump of~$u$ on $S_u$, and $\nu_u$ is the unit normal vector to $S_u$; finally, $\otimes$ denotes the dyadic product; 
	\item $L^p(\Omega;\mathbb R^{d\times N})$ is the set of matrix-valued $p$-integrable functions; 
	\item for $p\geq1$, $SD^p(\Omega)\coloneqq SBV(\Omega;\mathbb R^d)\times L^p(\Omega;\mathbb R^{d\times N})$ is the space of structured deformations $(g,G)$ (notice that $SD^1(\Omega)$ is the space $SD(\Omega)$ introduced in \cite{CF1997}); the norm in $SD(\Omega)$ is defined by $\norma{(g,G)}_{SD(\Omega)}\coloneqq\norma{g}_{BV(\Omega;\R{d})}+\norma{G}_{L^1(\Omega;\R{d\times N})}$;
	\item  $C$ represents a generic positive constant that may change from line to line.
\end{itemize}

\medskip 

A detailed exposition on $BV$ functions is presented in \cite{AFP}.

\medskip

The following result, whose proof may be found in \cite{FM1993}, will be used in the proof of Theorem \ref{GMthmHSD}.

\begin{lemma}\label{FM}
Let $\lambda$ be a non-negative Radon measure in $\mathbb R^N$. For $\lambda$ a.e. $x_0 \in \mathbb R^N$, for every $0 < \delta < 1$ and for every $\nu \in \mathbb S^{N-1}$, the following holds
$$\limsup_{\eps \to 0^+}\frac{\lambda(Q_\nu(x_0,\delta \eps))}{\lambda(Q_\nu(x_0,\eps))} \geq \delta^N,$$
so that
$$\lim_{\delta \to 1^-}\limsup_{\eps \to 0^+}\frac{\lambda(Q_\nu(x_0,\delta \eps))}{\lambda(Q_\nu(x_0,\eps))} = 1.$$
\end{lemma}

\subsection{Approximation theorem for hierarchical (first-order) structured deformations}

\begin{definition}\label{Def2.1}
	For $L\in\N{}$, $p\geq 1$ and $\Omega\subset\R{N}$ a bounded connected open set, we define
	$$HSD_L^p(\Omega)\coloneqq SBV(\Omega;\R{d})\times \underbrace{L^p(\Omega;\R{d\times N})\times\cdots\times L^p(\Omega;\R{d\times N})}_{L\text{-times}}$$
	the set of \emph{$(L+1)$-level (first-order) structured deformations} on $\Omega$. 
	\end{definition} 

In the case $L=1$ and $p =1$, this space was introduced and studied in \cite{CF1997}, where it was denoted by $SD(\Omega)$. In particular, the following approximation result was shown (see \cite[Theorem 2.12]{CF1997}).

\begin{theorem}[Approximation Theorem in $SD(\Omega)$]\label{appTHMCF} 
For every $(g,G)\in SD(\Omega)$, there exists a sequence 
$u_n$ in $SBV(\Omega;\R{d})$ which converges to $(g,G)$ in the sense that 
\begin{equation*}
u_n \to g\quad\text{in $L^1(\Omega; \R{d})$}\qquad \text{and}\qquad \nabla u_n \wsto G\quad\text{in $\cM(\Omega;\R{d\times N}).$}
\end{equation*}
\end{theorem}

We now present the definition of convergence of a sequence of $SBV$ functions to an $(L+1)$-level structured deformation $(g,G_1,\ldots,G_L)$ belonging to either $HSD_L(\Omega)$ or $HSD_L^p(\Omega)$. 

\begin{definition}\label{S000}
	Let $L\in\N{}$, let $p>1$, let $(g,G_1,\ldots,G_L)\in HSD_L^{p}(\Omega)$, and let $\N{L}\ni(n_1,\ldots,n_L)\mapsto u_{n_1,\ldots,n_L}\in SBV(\Omega;\R{d})$ be a (multi-indexed) sequence. We say that $u_{n_1,\ldots,n_L}$ converges in the sense of $HSD_L^p(\Omega)$ to $(g,G_1,\ldots,G_L)$ if
	\begin{itemize}
		\item[(i)] $\displaystyle\lim_{n_1\to+\infty}\cdots\lim_{n_L\to+\infty} u_{n_1,\ldots,n_L} = g$, with each of the iterated limits in the sense of $L^1(\Omega;\R{d})$;
		\item[(ii)] for all $\ell=1,\ldots,L-1$, 
		$\displaystyle \lim_{n_{\ell+1}\to+\infty}\cdots\lim_{n_L\to+\infty} u_{n_1,\ldots,n_L} \eqqcolon g_{n_1,\ldots,n_\ell} \in SBV(\Omega;\R{d})$ and 
		$$\lim_{n_1\to+\infty}\cdots\lim_{n_{\ell}\to+\infty} \nabla g_{n_1,\ldots,n_\ell}=G_{\ell},$$ 
		with each of the iterated limits in the sense of weak convergence in  $L^p(\Omega;\R{d\times N})$;
		\item[(iii)] $\displaystyle \lim_{n_1\to+\infty}\cdots\lim_{n_L\to+\infty} \nabla u_{n_1,\ldots,n_L} = G_L$ with each of the iterated limits in the sense of weak convergence in  $L^p(\Omega;\R{d\times N})$;
	\end{itemize}
	we use the notation $u_{n_1,\ldots,n_L}\pweak(g,G_1,\ldots,G_L)$ to indicate this convergence.
	
	In the case $p=1$, if $(g,G_1,\ldots,G_L)\in HSD_L^1(\Omega)$ and if the weak $L^p$ convergences above are replaced by weak-* convergences in $\cM(\Omega;\R{d\times N})$, then we say that $u_{n_1,\ldots,n_L}$ converges in the sense of $HSD_L^1(\Omega)$ to $(g,G_1,\ldots,G_L)$ and we use the notation $u_{n_1,\ldots,n_L} \weakst (g,G_1,\ldots,G_L)$ to indicate this convergence.
\end{definition}

The sequential application of the idea behind the Approximation Theorem~\ref{appTHMCF} provides the method for constructing a (multi-indexed) sequence $u_{n_1,\ldots,n_L}$ that approximates an $(L+1)$-level structured deformation $(g,G_1,\ldots,G_L)$. We thus obtain the following result, whose proof may be found in \cite{BMMOZ2022}.

\begin{theorem}[Approximation Theorem for $(L+1)$-level structured deformations]\label{appTHMh}
	Let  $(g,G_1,\ldots,G_L)$ belong to $HSD_L^p(\Omega)$. 
	Then there exists a sequence $(n_1,\ldots,n_L)\mapsto u_{n_1,\ldots,n_L}\in SBV(\Omega;\R{d})$
	converging to $(g,G_1,\ldots,G_L)$ in the sense of Definition~\ref{S000}.
\end{theorem}

\section{The global method}\label{gm}

Motivated by the Approximation Theorem \ref{appTHMh}, let $p \geq1$ and let $\mathcal F: HSD_L^p(\Omega) \times\mathcal 
O(\Omega)\to [0, +\infty]$ be a functional satisfying the following hypotheses
\begin{enumerate}
	\item[(H1)] for every $(g, G_1, \dots, G_L) \in HSD^p_L(\Omega)$, $\mathcal F(g, G_1,\dots, G_L;\cdot)$ is the restriction to 
	$\mathcal O(\Omega)$ of a Radon measure; 
	\item[(H2)] for every $O \in \mathcal O(\Omega)$, if  $p > 1$, $\mathcal F(\cdot, \cdot, \dots; O)$ is $HSD^p_L$-lower semicontinuous, i.e., 
	if $(g, G_1,\dots, G_L) \in HSD_L^p(\Omega)$ and $(g^n, G_1^n,\dots, G^n_L)$ $\in HSD_L^p(\Omega)$ are such that $g_n \to g$ in $L^1(\Omega;\mathbb R^{d} )$, $G^n_i\rightharpoonup G^i$ in $L^p(\Omega;\mathbb R^{d \times N})$, as $n \to +\infty$, for every $i=1,\dots L$, then
	$$
	\mathcal F(g, G_1,\dots, G_L;O)\leq \liminf_{n\to +\infty}\mathcal F(g^n,G_1^n,\dots, G^n_L;O);
	$$
	the same holds in the case $p=1$, replacing the weak convergences
	$G^n_i\rightharpoonup G^i$ in $L^p(\Omega;\mathbb R^{d \times N})$, as $n \to +\infty$, for every $i=1,\dots L$, with weak star convergences in the sense of measures $\cM(\Omega;\R{d\times N})$;
	\item[(H3)] for all $O \in \mathcal O(\Omega)$, $\mathcal F(\cdot, \ldots, \cdot;O)$ is local, that is, if $g= u$, $G_1= U_1$, $\dots$, $G_L=U_L$ a.e. in $O$, then 
	$\mathcal F(g, G_1,\dots, G_L;O)= \mathcal F(u, U_1,\dots, U_L;O)$;
	\item[(H4)] there exists a constant $C>0$ such that
	\begin{align*}
	 \frac{1}{C}\left(\sum_{i=1}^L \||G_i|^p\|_{L^1(O;\mathbb R^{d \times N})}+ |D g|(O)\right)&\leq \mathcal F(g, G_1,\dots, G_L;O)\\
	 &\leq
	  C\left(\mathcal L^N(O) +\sum_{i=1}^L \||G_i|^p\|_{L^1(O;\mathbb R^{d \times N})}+  |D g|(O)\right),
	\end{align*}
	for every $(g, G_1, \dots, G_L) \in HSD^p_L(\Omega)$ and every $O \in \mathcal O(\Omega)$.
\end{enumerate}

\begin{remark}\label{A1A2}
{\rm Due to hypotheses (H1) and (H4), given any $(u, U_1,\dots, U_L) \in HSD_L^p(\Omega)$ and any open sets $O_1 \subset \subset O_2 \subseteq \Omega$, it follows that 
\begin{align*}
\mathcal F(u, U_1,\dots, U_L;O_2) \leq & \, \mathcal F(u, U_1,\dots, U_L;O_1) \\
&+C\left(\mathcal L^N(O_2 \setminus O_1) +\sum_{i=1}^L \||U_i|^p\|_{L^1(O_2\setminus O_1;\mathbb R^{d \times N})}+  |D u|(O_2 \setminus O_1)\right).
\end{align*}
Indeed, for $\varepsilon > 0$ small enough, let $O_{\eps} : = \left\{x \in O_1 : {\rm dist}(x, \partial O_1) > \eps\right\}$ and notice that $O_2$ is covered by the union of the two open sets $O_1$ and $O_2 \setminus \overline{O_{\eps}}.$ Thus, by (H1) and (H4) we have
\begin{align*}
\mathcal F(u, U_1,\dots, U_L;O_2) \leq & \, \mathcal F(u, U_1,\dots, U_L;O_1) + \mathcal F(u, U_1,\dots, U_L;O_2 \setminus \overline{O_{\eps}}) \\
\leq & \, \mathcal F(u, U_1,\dots, U_L;O_1) \\
&\hspace{2cm}+
C\left(\mathcal L^N(O_2 \setminus  \overline{O_{\eps}}) +\sum_{i=1}^L \||U_i|^p\|_{L^1(O_2\setminus \overline{O_{\eps}};\mathbb R^{d \times N})}
+ |D u|(O_2 \setminus\overline{O_{\eps}})\right).
\end{align*}
To conclude the result it suffices to let $\eps \to 0^+$.
}
\end{remark}

\medskip

Given $(g, G_1, \dots, G_L)\in HSD_L^p(\Omega)$ and 
$O \in \mathcal O_{\infty}(\Omega)$, we introduce the space of test functions 
\begin{align}
	\mathcal C_{HSD^p_L}(g, G_1, \dots, G_L; O):=\left\{(u, U_1, \dots, U_L)\in HSD^p_L (\Omega): u=g \hbox{ in a neighbourhood of } \partial O, \right. \nonumber\\
	\left.\int_O (G_i-U_i) \,\de x=0, i=1,\dots, L  \right\}, 	\label{classM}
\end{align}
and we let $m: HSD^p_L(\Omega)\times \mathcal O_\infty(\Omega)$ be the functional defined by
\begin{align}\label{mdef}
m(g, G_1,\dots, G_L;O):=\inf\Big\{\mathcal F(u, U_1,\dots, U_L; O): (u, U_1,\dots, U_L)\in \mathcal C_{HSD^p_L}(g, G_1, \dots, G_L; O)\Big\}.
\end{align}

Following the ideas of the global method of relaxation introduced in \cite{BFM}, our aim in this section is to prove the theorem below. 

\begin{theorem}\label{GMthmHSD}
	Let $p \geq 1$ and let $\mathcal F: HSD_L^p(\Omega) \times\mathcal 
	O(\Omega)\to [0, +\infty]$ be a functional satisfying (H1)-(H4). 
Then
	\begin{align*}
	\mathcal F(u, U_1,\dots, U_L;O)= \int_O \!f(x,u(x), \nabla u(x), U_1(x),\dots, U_L(x))\,\de x +\int_{O \cap S_u}\!\!\!\!\!\Phi(x, u^+(x), u^-(x),\nu_u(x)) 
	\,\de \mathcal H^{N-1}(x),\end{align*}
where
\begin{align}\label{f}
f(x_0, a, \xi, B_1,\dots, B_L):=	\limsup_{\varepsilon \to 0^+}\frac{m(a+ \xi(\cdot-x_0), B_1,\dots, B_L; Q(x_0,\varepsilon))}{\varepsilon^N},
	\end{align}
\begin{align}\label{Phi}
	\Phi(x_0, \lambda, \theta, \nu):= \limsup_{\varepsilon \to 0^+}\frac{m(v_{\lambda, \theta,\nu}(\cdot-x_0), 0,\dots, 0; Q_{\nu}(x_0,\varepsilon))}{\varepsilon ^{N-1}},
	\end{align}
for all $x_0\in \Omega$, $ a, \theta,\lambda \in \mathbb R^d$, $\xi, B_1,\dots, B_L \in \mathbb R^{d \times N}$, $\nu \in \mathbb S^{N-1}$, and
where $0$ is the zero matrix in $\mathbb R^{d \times N}$ and 
$v_{\lambda,\theta, \nu}$ is defined by
$v_{\lambda,\theta, \nu}(x) := \begin{cases} \lambda, &\hbox{if } x\cdot \nu > 0\\
\theta, &\hbox{ if } x\cdot \nu \leq 0.\end{cases}$
\end{theorem}

\begin{remark}\label{traslinv} It follows immediately from the definitions given in \eqref{f} and in \eqref{Phi}, and from Theorem \ref{thm4.3FHP}, that if $\mathcal F$ is translation invariant in the first variable, i.e. if
$$\mathcal F(u+a, U_1, \dots, U_L;O)= \mathcal F(u, U_1,\dots, U_L;O),$$ 
for every $((u,U_1,\dots, U_l),O) \in HSD^p_L(\Omega)\times \mathcal O(\Omega)$ and for every $a\in \mathbb R^d$,
then the function $f$ in \eqref{f} does not depend on $a$ and the function $\Phi$ in \eqref{Phi} does not depend on $\lambda$ and $\theta$ but only on the difference $\lambda- \theta$.
Indeed, in this case we conclude that
$$f(x_0, a, \xi, B_1,\dots, B_L) = f(x_0, 0, \xi, B_1,\dots, B_L) \quad \mbox{and}
\quad
\Phi(x_0, \lambda, \theta, \nu) = \Phi(x_0, \lambda - \theta,0, \nu),$$
for all $x_0\in \Omega$, $ a, \theta,\lambda \in \mathbb R^d$, $\xi, B_1,\dots, B_L \in \mathbb R^{d \times N}$ and $\nu \in \mathbb S^{N-1}$. With an abuse of notation
we write
$$f(x_0,\xi, B_1,\dots, B_L) = f(x_0, 0, \xi, B_1,\dots, B_L)  \quad \mbox{and}
\quad \Phi(x_0, \lambda - \theta, \nu) = \Phi(x_0, \lambda - \theta,0, \nu).$$ 
\end{remark}

The proof of Theorem \ref{GMthmHSD} is based on several auxiliary results and follows the reasoning presented in \cite[Theorem 3.7]{BFM} and \cite[Theorem 4.6]{FHP}. For this reason we don't provide the arguments in full detail but point out only the main differences that arise in our setting. We start by proving the following lemma which is used to obtain Theorem \ref{thm4.3FHP}.

\begin{lemma}\label{estHSDL}
Assume that (H1) and (H4) hold. For any $(u, U_1, \dots, U_L)\in HSD^p_L(\Omega)$ it follows that  
\begin{itemize}
	\item if $p>1$, 
	$$
	\limsup_{\delta \to 0^+} m(u, U_1, \dots, U_L;Q_\nu(x_0,(1-\delta) r))
	\leq m(u, U_1, \dots, U_L;Q_\nu(x_0,r)),
	$$
	where
	$Q_\nu(x_0, r)$ is any cube centred at $x_0$ with side length $r$, two faces orthogonal to $\nu$ and contained in $\Omega$;
	\item if $p=1$, 
	$$
	\limsup_{\delta \to 0^+} m(u, U_1, \dots, U_L;O_\delta)
	\leq m(u, U_1, \dots, U_L; O),
	$$
	where $ O_\delta=\{x \in O: {\rm dist}(x, \partial O) > \delta\}$ and $O \in \mathcal O(\Omega)$.
	\end{itemize}
\end{lemma}

\begin{proof}  
Suppose first that $p > 1$. Without loss of generality we can assume that $x_0=0$, $r=1$, $\nu=\bf e_1$ and $Q \subset \Omega$.  For every $\varepsilon >0$ there exists $(v, V_1,\dots V_L)\in \mathcal C_{HSD^p_L}(u, U_1,\dots U_L;Q)$ such that 
\begin{equation}\label{uQ}
\mathcal F(v, V_1,\dots, V_L;Q)\leq m(u, U_1,\dots, U_L;Q)+ \varepsilon.
\end{equation}
Let $0 < \delta < 1$ be small enough so that $u = v$ in a neighbourhood of $Q \setminus Q(1-2\delta)$, and let $\delta < \alpha(\delta) < 2\delta$ be such that 
$\displaystyle \lim_{\delta \to 0^+}\alpha(\delta) = 0$, 
$Q(1-\alpha(\delta)) \subset \subset Q(1-\delta)$ and
\begin{equation}\label{bound}
\dfrac{\mathcal L^N(Q \setminus Q(1-\alpha(\delta)))}{\mathcal L^N(Q(1-\delta) \setminus Q(1-\alpha(\delta)))} \leq C,
\end{equation}
where the constant $C$ depends only on the space dimension $N$ and is, therefore, independent of $\delta$.

For every $i =1,\dots, L$, define
$$
\overline V_i= \begin{cases}
V_i, &\hbox{ in }Q(1-\alpha(\delta))\\
\dfrac{1}{\mathcal L^N(Q(1-\delta) \setminus Q(1-\alpha(\delta)))}\left(\displaystyle \int_{Q(1-\delta)}U_i \, \de x- \int_{Q(1-\alpha(\delta))}V_i \, \de x\right), &\hbox{ in }Q(1-\delta)\setminus Q(1-\alpha(\delta))\\*[5mm]
U_i, &\hbox{ in } \Omega\setminus Q(1-\delta)	
\end{cases}
$$
and
$$
\overline v= \begin{cases} 
v, &\hbox{ in } Q(1-\alpha(\delta)) \\
u, &\hbox{ in } \Omega\setminus Q(1-\alpha(\delta)).
\end{cases}
$$
It is easily verified that $(\overline v, \overline V_1,\dots, \overline V_L)\in \mathcal C_{HSD^p_L}(u, U_1,\dots, U_L;Q(1-\delta))$. Thus, by Remark \ref{A1A2}, by (H1) and by \eqref{uQ}, we have
\begin{align}\label{uQ1-delta}
m(u, U_1,\dots, U_L;Q(1-\delta))&\leq 
\mathcal F(\overline v, \overline V_1, \dots, \overline V_L;Q(1-\delta))\nonumber\\
& \leq \mathcal F(v, V_1, \dots, V_L; Q(1-\alpha(\delta))) 
+ C\Big[\mathcal L^N(Q_{1- \delta} \setminus Q(1-\alpha(\delta))) \nonumber\\
& \hspace{1,3cm} + |D \overline v|(Q(1-\delta) \setminus Q(1-\alpha(\delta)))
 +  \sum_{i=1}^L\||\overline V_i|^p\|_{L^1(Q(1-\delta) \setminus Q(1-\alpha(\delta));\mathbb R^{d \times N})}\Big]\nonumber\\
&\leq \mathcal F(v, V_1, \dots, V_L;Q) + 
C\Big[\mathcal L^N(Q_{1- \delta} \setminus Q(1-\alpha(\delta))) \nonumber\\
& \hspace{1,3cm} + |D \overline v|(Q(1-\delta) \setminus Q(1-\alpha(\delta)))
 +  \sum_{i=1}^L\||\overline V_i|^p\|_{L^1(Q(1-\delta) \setminus Q(1-\alpha(\delta));\mathbb R^{d \times N})}\Big]\nonumber\\
&\leq   m(u, U_1,\dots, U_L;Q)+ \varepsilon + 
C\Big[\mathcal L^N(Q_{1- \delta} \setminus Q(1-\alpha(\delta))) \nonumber\\
& \hspace{1,3cm} + |D \overline v|(Q(1-\delta) \setminus Q(1-\alpha(\delta))) 
 +  \sum_{i=1}^L\||\overline V_i|^p\|_{L^1(Q(1-\delta) \setminus Q(1-\alpha(\delta));\mathbb R^{d \times N})}\Big].
\end{align}
Clearly, $\displaystyle \lim_{\delta \to 0^+}\mathcal L^N(Q_{1- \delta} \setminus Q(1-\alpha(\delta))) = 0$ and, since $u = v$ on $\partial Q(1-\alpha(\delta))$, it also follows that 
$$\displaystyle \lim_{\delta \to 0^+}|D \overline v|(Q(1-\delta) \setminus Q(1-\alpha(\delta))) =0.$$
On the other hand, for every $i = 1, \ldots,L$, we have
\begin{align}\label{pnorm}
&\||\overline V_i|^p\|_{L^1(Q(1-\delta) \setminus Q(1-\alpha(\delta));\mathbb R^{d \times N})}  =\frac{1}{(\mathcal L^N(Q(1-\delta) \setminus Q(1-\alpha(\delta))))^{p-1}} \left|\int_{Q(1-\delta)} U_i \, \de x - \int_{Q(1-\alpha(\delta))} V_i \, \de x\right|^p \nonumber\\
& \hspace{2cm}= \frac{1}{(\mathcal L^N(Q(1-\delta) \setminus Q(1-\alpha(\delta))))^{p-1}} \left|\int_{Q(1-\alpha(\delta))} U_i - V_i \, \de x + \int_{Q(1-\delta) \setminus Q(1-\alpha(\delta))} U_i \, \de x\right|^p \nonumber \\
& \hspace{2cm}\leq \frac{C}{(\mathcal L^N(Q(1-\delta) \setminus Q(1-\alpha(\delta))))^{p-1}} \left(\left|\int_{Q(1-\alpha(\delta))} U_i - V_i \, \de x\right|^p + \left|\int_{Q(1-\delta) \setminus Q(1-\alpha(\delta))} U_i \, \de x\right|^p\right).
\end{align}
Recalling that $\displaystyle \int_Q U_i - V_i \, \de x = 0$, $\forall i = 1, \ldots,L$,
the first term in \eqref{pnorm} can be estimated by using H\"older's inequality yielding
\begin{align*}
&\frac{C}{(\mathcal L^N(Q(1-\delta) \setminus Q(1-\alpha(\delta))))^{p-1}} \left|\int_{Q(1-\alpha(\delta))} U_i - V_i \, \de x\right|^p \nonumber \\
&\hspace{2cm}= \frac{C}{(\mathcal L^N(Q(1-\delta) \setminus Q(1-\alpha(\delta))))^{p-1}} \left|\int_{Q\setminus Q(1-\alpha(\delta))} U_i - V_i \, \de x\right|^p \nonumber \\
&\hspace{2cm}\leq \frac{C}{(\mathcal L^N(Q(1-\delta) \setminus Q(1-\alpha(\delta))))^{p-1}}
\|U_i - V_i\|^p_{L^p(Q \setminus Q(1-\alpha(\delta));\mathbb R^{d \times N})}
(\mathcal L^N(Q \setminus Q(1-\alpha(\delta))))^{p-1}.	
\end{align*}	
By \eqref{bound} and the fact that $\displaystyle \lim_{\delta \to 0^+}\mathcal L^N(Q\setminus Q(1-\alpha(\delta))) = 0$ we conclude that
$$\lim_{\delta \to 0^+}	\frac{C}{(\mathcal L^N(Q(1-\delta) \setminus Q(1-\alpha(\delta))))^{p-1}} 
\left|\int_{Q(1-\alpha(\delta))} U_i - V_i \, \de x\right|^p = 0.$$
Regarding the second term in \eqref{pnorm}, a similar argument using H\"older's inequality leads to
\begin{align*}
&\lim_{\delta \to 0^+}
\frac{C}{(\mathcal L^N(Q(1-\delta) \setminus Q(1-\alpha(\delta))))^{p-1}}
\left|\int_{Q(1-\delta) \setminus Q(1-\alpha(\delta))} U_i \, \de x\right|^p \\
\smallskip
&\hspace{7cm}\leq \lim_{\delta \to 0^+}C\|U_i\|^p_{L^p(Q(1-\delta) \setminus Q(1-\alpha(\delta));\mathbb R^{d \times N})} = 0.
\end{align*}
Therefore, from \eqref{uQ1-delta}, we obtain
$$\limsup_{\delta\to 0^+}m(u, U_1,\dots, U_L;Q(1-\delta))\leq 
m(u, U_1,\dots, U_L;Q)+ \varepsilon$$
and it suffices to let $\varepsilon \to 0^+$ to complete the proof in the case 
$p > 1$.

When $p = 1$ the proof is similar and we omit the details. In this case the estimate of the last term in \eqref{uQ1-delta} is simpler and does not require the use of H\"older's inequality. Also, in this case, more general sets other than cubes may be considered as there is no need to use inequality \eqref{bound} (see also \cite{FHP}).
\end{proof}

\medskip

Following \cite{BFM, BFLM}, for a fixed $(u,U_1,\dots,U_L) \in  HSD_L^p(\Omega)$,
we set $\mu := \mathcal L^N\lfloor \Omega + |D^s u|$ and we define
$$\mathcal O^\star(\Omega):= \{Q_\nu (x, \varepsilon) : x \in \Omega, \nu \in \mathbb S^{N-1}, \varepsilon > 0\},$$
and, for $O \in \mathcal O(\Omega)$ and $\delta > 0$, we let
\begin{align*}
	m^\delta(u,U_1,\dots, U_L;O) := \inf \Big\{\sum_{i=1}^\infty m(u, U_1,\dots, U_L; Q_i): Q_i \in \mathcal O^\star(\Omega), Q_i \subseteq O, Q_i \cap Q_j= \emptyset \; {\rm if } \; i \neq j,  \\
	 {\rm diam}(Q_i) < \delta, \mu\Big(O \setminus \bigcup_{i=1}^\infty Q_i\Big)=0\Big\}.
\end{align*}
Since $\delta \mapsto m^\delta(u,U_1,\dots, U_L;O)$ is an increasing function, we now define 
$$
m^\star (u, U_1, \dots, U_L;O):= \sup_{\delta >0} m^\delta (u, U_1, \dots, U_L;O)= \lim_{\delta \to 0^+} m^\delta(u, U_1,\dots, U_L;O).
$$

Adapting the reasoning given in \cite[Lemma 4.2 and Theorem 4.3]{FHP}, with an even easier argument due to our hypothesis (H2) and to the fact that our fields $u$ have bounded variation, we obtain the two results below.

\begin{lemma}
	\label{Lemma 4.2FHP} Let $p \geq 1$ and
	assume that (H1)-(H4) hold. Then, 
	for all $(u, U_1,\dots, U_L)\in HSD_L^p(\Omega)$ and all 
	$O \in \mathcal O(\Omega)$,
	we have
	$$
	\mathcal F(u, U_1, \dots, U_L;O) = m^\star(u, U_1,\dots, U_L; O).
	$$
\end{lemma}

\medskip

\begin{theorem}\label{thm4.3FHP}
	Let $p\geq 1$ and assume that hypotheses (H1), (H2) and (H4) hold. Then, for every $\nu \in \mathbb S^{N-1}$ and for every $(u,U_1, \dots, U_L) \in HSD_L^p(\Omega)$, 
	we have 
	$$
	\lim_{\varepsilon \to 0^+}\frac{{\mathcal F}(u, U_1,\dots, U_L; Q_\nu(x_0, \varepsilon))}{\mu( Q_\nu(x_0,\varepsilon))}= \lim_{\varepsilon \to 0^+} \frac{m (u, U_1, \dots, U_L; Q_\nu(x_0, \varepsilon))}{\mu (Q_\nu(x_0,\varepsilon))}
	$$
	for $\mu$-a.e. $x_0 \in \Omega$, where $\mu:=\mathcal L^N\lfloor \Omega + |D^s u|.$
	\end{theorem}

\medskip

We now present the proof of our main result of this section.

\begin{proof}[Proof of Theorem \ref{GMthmHSD}]
	We begin by proving that, for $\mathcal L^N$- a.e. $x_0 \in \Omega$,
	\begin{align}\label{fproof}
	\frac{d \mathcal F(u, U_1, \dots, U_L;\cdot)}{d \mathcal L^N} (x_0)=	
	f(x_0, u(x_0), \nabla u(x_0), U_1(x_0),\dots, U_L(x_0)).
	\end{align}

Let $x_0$ be a fixed point in $\Omega$ satisfying the following properties
\begin{align}
&\lim_{\varepsilon \to 0^+}\frac{1}{\varepsilon}\ave_{Q(x_0,\varepsilon)}|u(x) - u(x_0) - \nabla u(x_0)(x-x_0)| \, \de x = 0;\label{aproxdif}\\
& \lim_{\varepsilon \to 0^+}\frac{1}{\varepsilon^N}|Du|(Q(x_0,\varepsilon)) = |\nabla u(x_0)|, \quad \lim_{\varepsilon \to 0^+}\frac{1}{\varepsilon^N}|D^su|(Q(x_0,\varepsilon)) = 0;\label{DuDsu}\\
& \lim_{\varepsilon \to 0^+}\ave_{Q(x_0,\varepsilon)}|U_i(x) - U_i(x_0)| \, \de x = 0, 
\forall i = 1, \ldots,L;\label{LebptUj}\\
& \frac{d \mathcal F(u, U_1, \dots, U_L;\cdot)}{d \mathcal L^N} (x_0) =
\lim_{\varepsilon \to 0^+}\frac{\mathcal F(u, U_1, \dots, U_L;Q(x_0,\varepsilon))}{\varepsilon^N} =
\lim_{\varepsilon \to 0^+}\frac{m(u, U_1, \dots,U_L;Q(x_0,\varepsilon))}{\varepsilon^N};\label{Fmu}\\
& \frac{d \mathcal F(v_a, U_1(x_0), \dots, U_L(x_0);\cdot)}{d \mathcal L^N} (x_0) =
\lim_{\varepsilon \to 0^+}\frac{m(v_a, U_1(x_0), \dots,U_L(x_0);Q(x_0,\varepsilon))}{\varepsilon^N};\label{Fmva}
\end{align}
where we are denoting by $v_a$ the function defined in $\Omega$ by 
$v_a(x):= u(x_0)+ \nabla u(x_0)(x-x_0)$. It is well known that the above properties hold for $\mathcal L^N$- a.e. point $x_0$ in $\Omega$, taking also in consideration Theorem \ref{thm4.3FHP} in \eqref{Fmu} and \eqref{Fmva}.

Having fixed $x_0$ as above, let $\delta \in (0,1)$ and let $\varepsilon >0$ be small enough so that $Q(x_0,\varepsilon)\subset \Omega$. Given the definition of the density $f$ in \eqref{f}, due to \eqref{Fmu} and \eqref{Fmva},
we want to show that
\begin{align}\label{toestimate}
	\lim_{\varepsilon \to 0^+}
	\frac{m(v_a,U_1(x_0),\dots U_L(x_0);Q (x_0, \varepsilon))}
	{\mathcal L^N (Q (x_0, \varepsilon))}- \lim_{\varepsilon \to 0^+}\frac{m(u,U_1,\dots U_L; Q (x_0, \varepsilon))}{\mathcal L^N (Q (x_0, \varepsilon))} = 0.
\end{align}

Let $(\widetilde u, \widetilde U_1, \dots, \widetilde U_L) \in \mathcal C_{HSD^p_L}(v_a, U_1(x_0), \dots, U_L(x_0); Q(x_0, \delta \varepsilon))$ be such that
\begin{equation}\label{infmva}
\varepsilon^{N+1} + m(v_a, U_1(x_0),\dots, U_L(x_0); Q(x_0, \delta \varepsilon)) \geq \mathcal F(\widetilde u, \widetilde U_1, \dots, \widetilde U_L; Q(x_0, \delta \varepsilon)).
\end{equation}
Then, as $\widetilde u = v_a$ on $\partial Q(x_0,\varepsilon)$, we have
\begin{align}\label{trvaest}
|{\rm tr}\, u-{\rm tr}\,\widetilde u|(\partial Q(x_0,\delta \varepsilon)):= \int_{\partial Q(x_0,\delta\varepsilon)} |\widetilde u(x)- u(x)| \,\de \mathcal H^{N-1}(x)=\int_{\partial Q(x_0,\delta\varepsilon)} |v_a(x)- u(x)|\, \de \mathcal H^{N-1}(x).
\end{align}

Let $\delta' \in (\delta, 1)$ be such that $Q(x_0,\delta \varepsilon) \subset \subset Q(x_0,\delta' \varepsilon)$ and define
$$
\widetilde v_\varepsilon:=\left\{
\begin{array}{ll}
\widetilde u, &\hbox{ in }Q(x_0,\delta \varepsilon),\\
u, &\hbox{ in }\Omega\setminus Q(x_0,\delta \varepsilon)
\end{array}
\right.
$$
and, for every $i\in \{1,\dots ,L\}$, let
$$\widetilde V^i_\varepsilon(x):= \left\{
\begin{array}{ll}
\widetilde U_i(x), & \hbox{ in }Q(x_0,\delta \varepsilon),\\
\displaystyle \frac{1}{\mathcal L^N(Q(x_0,\varepsilon)\setminus Q(x_0,\delta \varepsilon))}\left[\int_{Q(x_0,\varepsilon)} U_i(x)\, \de x- \int_{Q(x_0,\delta \varepsilon)} U_i(x_0)\,\de x\right], &\hbox{ in }\Omega\setminus Q(x_0,\delta \varepsilon).
\end{array}\right.
$$
Recall that  $\displaystyle \int_{Q(x_0,\delta \varepsilon)} \widetilde U_i(x) \,\de x = \int_{Q(x_0,\delta \varepsilon)} U_i(x_0) \,\de x = U_i(x_0)(\delta \varepsilon)^N$, for every $i\in \{1,\dots, L\}$, so it is easy to see that
$(\widetilde v_\varepsilon, \widetilde V^1_{\varepsilon},\dots, \widetilde V^L_{\varepsilon})\in \mathcal C_{HSD^p_L}(u, U_1, \dots, U_L; Q(x_0,\varepsilon))$.
Hence, by Remark \ref{A1A2}, (H4) and \eqref{infmva}, we have
\begin{align}\label{mest}
	m(u,U_1,\dots, U_L;Q(x_0,\varepsilon))&\leq \mathcal F(\widetilde v_\varepsilon,\widetilde V^1_\varepsilon,\dots, \widetilde V^L_\varepsilon;Q(x_0,\varepsilon))\nonumber\\
	&\leq \mathcal F(\widetilde v_\varepsilon, \widetilde V^1_\varepsilon,\dots, \widetilde V^L_\varepsilon;Q(x_0,\delta'\varepsilon))+ \mathcal F(\widetilde v_\varepsilon, \widetilde V^1_\varepsilon,\dots, \widetilde V^L_\varepsilon; Q(x_0,\varepsilon)\setminus \overline{ Q (x_0,\delta\varepsilon)})\nonumber \\
	&\leq \mathcal F(\widetilde v_\varepsilon, \widetilde V^1_\varepsilon,\dots, \widetilde V^L_\varepsilon;Q(x_0,\delta\varepsilon))+ \mathcal F(\widetilde v_\varepsilon, \widetilde V^1_\varepsilon,\dots, \widetilde V^L_\varepsilon; Q(x_0,\varepsilon)\setminus \overline{ Q (x_0,\delta\varepsilon)})\nonumber \\
	& \hspace{1cm} + C\Big(\mathcal L^N(Q(x_0,\delta'\varepsilon) \setminus Q(x_0,\delta\varepsilon)) + |D \widetilde v_\varepsilon|(Q(x_0,\delta'\varepsilon)\setminus Q(x_0,\delta \varepsilon))\nonumber \\
	& \hspace{2cm} + \sum_{i=1}^L \int_{Q(x_0,\delta'\varepsilon)\setminus Q(x_0,\delta \varepsilon)} |\widetilde V^i_\varepsilon  |^p \,\de x\Big)\nonumber \\
	&\leq \mathcal F(\widetilde u, \widetilde U_1,\dots, \widetilde U_L;Q(x_0,\delta\varepsilon)) + C\Big(\mathcal L^N(Q(x_0,\varepsilon) \setminus Q(x_0,\delta\varepsilon))\nonumber \\ 
	& \hspace{2cm} + \sum_{i=1}^L \int_{Q(x_0,\varepsilon)\setminus Q(x_0,\delta \varepsilon)} |\widetilde V^i_\varepsilon  |^p \,\de x 
	 + |D \widetilde v_\varepsilon|(Q(x_0,\varepsilon)\setminus Q(x_0,\delta \varepsilon))\Big)\nonumber \\
	&\leq \varepsilon^{N+1}+ m(v_a, U_1(x_0),\dots, U_L(x_0); Q(x_0,\delta\varepsilon))\nonumber \\
	& \hspace{1cm}+ C\Big( \varepsilon^N (1-\delta^N) + |D u|(Q(x_0,\varepsilon)\setminus \overline{Q(x_0,\delta \varepsilon)})+ |{\rm tr}\,\widetilde u-{\rm \tr}\, u|(\partial Q(x_0,\delta\varepsilon))\nonumber \\
	&\hspace{2cm}+\sum_{i=1}^L \int_{Q(x_0,\varepsilon)\setminus Q(x_0,\delta \varepsilon)} |\widetilde V^i_\varepsilon  |^p \,\de x\Big). 
\end{align}

We observe that for every $i\in \{1,\dots, L\}$ we have
\begin{align}\label{mest2}
& \int_{Q(x_0,\varepsilon)\setminus Q(x_0,\delta \varepsilon)} |\widetilde V^i_\varepsilon  |^p \,\de x \leq 
\frac{1}{\varepsilon ^{N(p-1)}(1-\delta^N)^{p-1}}\left|\int_{Q(x_0,\varepsilon)} U_i(x) \,\de x - \int_{Q(x_0,\delta \varepsilon)} U_i(x_0) \,\de x\right|^p  \nonumber\\
&\hspace{0,3cm}\leq\frac{C}{\varepsilon ^{N(p-1)}(1-\delta^N)^{p-1}}\left( \left|\int_{Q(x_0,\varepsilon)\setminus Q(x_0,\delta \varepsilon)} U_i(x) \, \de x\right|^p+ \left|\int_{Q(x_0,\delta\varepsilon)}(U_i(x)-U_i(x_0)) \,\de x\right|^p\right) \nonumber \\
&\hspace{0,3cm}\leq\frac{C \varepsilon^{Np}}{\varepsilon ^{N(p-1)}(1-\delta^N)^{p-1}}\left( \left|\ave_{Q(x_0,\varepsilon)}\hspace{-0,1cm} U_i(x) \,\de x- \delta^N \ave_{Q(x_0,\delta\varepsilon)}\hspace{-0,1cm}U_i(x)\,\de x \right|^p+ \left|\delta^N\ave_{Q(x_0,\delta\varepsilon)}\hspace{-0,1cm}U_i(x)-U_i(x_0)\,\de x\right|^p\right). 
\end{align}
Thus, to obtain \eqref{toestimate}, taking into account \eqref{mest}, \eqref{mest2} and Lemma \ref{FM}, we have
\begin{align}
& \lim_{\varepsilon \to 0^+}\frac{m(u,U_1,\dots, U_L; Q (x_0, \varepsilon))}{\mathcal L^N (Q (x_0, \varepsilon))} - \lim_{\varepsilon \to 0^+}
	\frac{m(v_a,U_1(x_0),\dots, U_L(x_0);Q (x_0, \varepsilon))}
	{\mathcal L^N (Q (x_0, \varepsilon))}\nonumber\\
&\leq \lim_{\varepsilon \to 0^+}\frac{m(u,U_1,\dots, U_L; Q (x_0, \varepsilon))}{\varepsilon^N} - \limsup_{\delta\to 1^-}\lim_{\varepsilon \to 0^+}
	\frac{m(v_a,U_1(x_0),\dots, U_L(x_0);Q (x_0, \delta\varepsilon))}
	{\varepsilon^N}\nonumber\\
	&\hspace{1cm}\leq\limsup_{\delta \to 1^-}\limsup_{\varepsilon \to 0^+} \left(\varepsilon + C(1-\delta^N)+ \frac{|D u|(Q(x_0,\varepsilon)\setminus \overline{Q(x_0,\delta \varepsilon)}+ |{\rm tr}\,\widetilde u-{\rm \tr}\, u|(\partial Q(x_0,\delta\varepsilon))}{\varepsilon^N}+ \right. \nonumber\\
	&\left.\hspace{4cm}\frac{C}{(1-\delta^N)^{p-1}} \sum_{i=1}^L|U_i(x_0)-\delta^N U_i(x_0)| ^p\right)\label{mest3}
\end{align}
where in the last line we have used the fact that $x_0$ is a Lebesgue point for $U_i$, see \eqref{LebptUj}.

Using \eqref{DuDsu} and \cite[(5.79)]{AFP} yields
\begin{align}\label{uapproxest}
\limsup_{\delta \to 1^-}\limsup_{\varepsilon \to 0^+}\frac{|D u|(Q(x_0,\varepsilon)\setminus \overline{Q(x_0,\delta \varepsilon)})}{\varepsilon^N}\leq \lim_{\delta \to 1^-}|\nabla u(x_0)|(1-\delta^N)=0.
\end{align}
On the other hand, by \eqref{trvaest} and a change of variables, we can apply \cite[Lemma 2.3]{BFM} to conclude that 
\begin{align}
\limsup_{\varepsilon \to 0^+} \frac{|{\rm tr}\,\widetilde u-{\rm \tr}\, u|(\partial Q(x_0,\delta\varepsilon))}{\varepsilon^N}&=\limsup_{\varepsilon \to 0^+} \delta^N\frac{|{\rm tr}\,v_a-{\rm \tr}\, u|(\partial Q(x_0,\delta\varepsilon))}{\delta^N\varepsilon^N}\nonumber\\
&=\limsup_{\varepsilon \to 0^+}\delta^N\int_{\partial Q} | {\rm tr ( u_{\varepsilon \delta}}- \nabla u(x_0)y)|\, \de\mathcal H^{N-1}(y) =0,\label{esttr}
\end{align}
since, denoting by $\displaystyle u_{\varepsilon \delta}(y):=\frac{u(x_0+\delta\varepsilon y)-u(x_0)}{\delta\varepsilon}$, it follows from \eqref{aproxdif} and \eqref{DuDsu} that $u_{\varepsilon \delta} \to \nabla u(x_0)y$ in $L^1(Q;\mathbb R^d)$ and 
$|Du_{\varepsilon\delta}|(Q)\to |\nabla u(x_0)|$, as $\varepsilon \to 0^+$.

Taking into account \eqref{mest3}, \eqref{uapproxest} and \eqref{esttr} we conclude that
$$\lim_{\varepsilon \to 0^+}\frac{m(u,U_1,\dots, U_L; Q (x_0, \varepsilon))}{\varepsilon^N} \leq \lim_{\varepsilon \to 0^+}
	\frac{m(v_a,U_1(x_0),\dots, U_L(x_0);Q (x_0, \varepsilon))}
	{\varepsilon^N}.$$
Interchanging the roles of $(u,U_1,\dots, U_L)$ and $(v_a,U_1(x_0),\dots, U_L(x_0))$, the reverse inequality is proved in a similar fashion. 
This completes the proof of \eqref{fproof}.

Next we want to prove that, for $\mathcal H^{N-1}$- a.e $x_0 \in S_u$,
\begin{align*}
\frac{d \mathcal F(u, U_1,\dots, U_L; \cdot)}{d \mathcal H^{N-1}\lfloor{S_u}}(x_0) = \Phi(x_0, u^+(x_0), u^-(x_0), \nu_u(x_0)).
\end{align*}

For simplicity of notation we denote by $\nu$ the unit vector $\nu_u$ and by $v_j$ the function defined in $\Omega$ by
$$v_j(x) = v_{u^+(x_0),u^-(x_0),\nu(x_0)}(x-x_0):=\left\{
\begin{array}{ll}
u^+(x_0) &\hbox{ if } (x-x_0)\cdot \nu(x_0) > 0,\\
u^-(x_0) &\hbox{ if }(x-x_0)\cdot \nu(x_0) \leq 0. 
\end{array}\right.
$$

It is well known that, for $\mathcal H^{N-1}$ a.e $x_0 \in S_u$, the following hold
\begin{align}
& \lim_{\varepsilon \to 0^+}\ave_{Q_{\nu}(x_0,\varepsilon)}|u(x) - v_j(x)| \, \de x =0;\label{jumppt}\\
& \lim_{\varepsilon \to 0^+}\frac{1}{\varepsilon^{N-1}}|Du|(Q_{\nu}(x_0,\varepsilon)) = |[u](x_0)|;\label{Du}\\
& \frac{d \mathcal F(u, U_1, \dots, U_L;\cdot)}{d \mathcal H^{N-1}\lfloor{S_u}} (x_0) =
\lim_{\varepsilon \to 0^+}\frac{\mathcal F(u, U_1, \dots, U_L;Q_{\nu}(x_0,\varepsilon))}{\varepsilon^{N-1}} =
\lim_{\varepsilon \to 0^+}\frac{m(u, U_1, \dots,U_L;Q_{\nu}(x_0,\varepsilon))}{\varepsilon^{N-1}};\label{jFmu}\\
& \frac{d \mathcal F(v_j, 0, \dots,0;\cdot)}{d \mathcal H^{N-1}\lfloor{S_u}} (x_0) =
\lim_{\varepsilon \to 0^+}\frac{m(v_j,0,\dots,0;Q_{\nu}(x_0,\varepsilon))}{\varepsilon^{N-1}};\label{Fmvj}\\
& \lim_{\varepsilon \to 0^+}\frac{1}{\varepsilon^{N-1}}
\int_{Q_{\nu}(x_0,\varepsilon)}|U_i(x)|^p \, \de x = 0, \; \forall i = 1, \ldots,L;
\label{Uip}
\end{align}
where  Theorem \ref{thm4.3FHP} was used in \eqref{jFmu} and \eqref{Fmvj}.

Let $x_0$ be a fixed point in $\Omega$ satisfying the above properties, 
let $\delta \in (0,1)$ and let $\varepsilon >0$ be small enough so that $Q_{\nu}(x_0,\varepsilon)\subset \Omega$. Also, let $\delta' \in (\delta,1)$ be such that $Q_\nu(x_0,\delta\eps) \subset \subset Q_\nu(x_0,\delta'\eps)$.

Given the definition of the density $\Phi$ in \eqref{Phi}, due to \eqref{jFmu} and \eqref{Fmvj},
we want to show that
\begin{align}\label{mestjump}
\lim_{\varepsilon \to 0^+}
\frac{m(v_j,0,\dots, 0 ;Q_\nu (x_0, \varepsilon))}
{\varepsilon^{N-1}}- \lim_{\varepsilon \to 0^+}\frac{m(u,U_1,\dots, U_L; Q_\nu(x_0, \varepsilon))}{\varepsilon^{N-1}} =0
\end{align}
where $0$ is the null function from $\Omega$ to $\mathbb R^{d\times N}$.

To this end, let $(\widetilde u, \widetilde U_1, \dots, \widetilde U_L) \in \mathcal C_{HSD^p_L}(v_j, 0, \dots, 0); Q_\nu(x_0, \delta \varepsilon))$ be such that
\begin{equation}\label{infmvj}
\varepsilon^{N} + m(v_j, 0,\dots, 0; Q_\nu(x_0, \delta \varepsilon)) \geq \mathcal F(\widetilde u, \widetilde U_1, \dots, \widetilde U_L; Q_\nu(x_0, \delta \varepsilon)).
\end{equation}

Notice that, as $\widetilde u = v_j$ on $\partial Q_\nu(x_0, \delta \varepsilon)$, we have
\begin{equation}\label{trvj}
|{\rm tr} u-{\rm tr}\widetilde u|(\partial Q_\nu(x_0,\delta \varepsilon))
= \int_{\partial Q_\nu(x_0,\delta\varepsilon)} 
|{\rm tr}(\widetilde u(x)- u(x))| \, \de \mathcal H^{N-1}(x)
=\int_{\partial Q_\nu(x_0,\delta\varepsilon)} 
|{\rm tr}(v_j(x)- u(x))| \, \de \mathcal H^{N-1}(x).
\end{equation}

Define
$$
\widetilde v_\varepsilon:=\left\{
\begin{array}{ll}
\widetilde u &\hbox{ in }Q_\nu(x_0,\delta \varepsilon),\\
u &\hbox{ in } \Omega \setminus Q_\nu(x_0,\delta \varepsilon)
\end{array}
\right.
$$
and, for every $i\in \{1,\dots ,L\}$, let
$$\widetilde V^i_\varepsilon(x):= \left\{
\begin{array}{ll}
\widetilde U_i(x) & \hbox{ in }Q_\nu(x_0,\delta \varepsilon),\\
\dfrac{1}{\mathcal L^N(Q_\nu(x_0,\varepsilon)\setminus Q_\nu(x_0,\delta \varepsilon))}\displaystyle \int_{Q_\nu(x_0,\varepsilon)} U_i(x) \, \de x &\hbox{ in } \Omega \setminus Q_\nu(x_0,\delta \varepsilon).
\end{array}\right.
$$
Recall that, for every $i\in \{1,\dots, L\}$, we have 
$ \displaystyle \int_{Q_\nu(x_0,\delta \varepsilon)} \widetilde U_i(x) \, \de x=0$.
Thus, 
$(\widetilde v_\varepsilon, \widetilde V^1_{\varepsilon},\dots, \widetilde V^L_{\varepsilon})$ belongs to the class of admissible test functions 
$\mathcal C_{HSD^p_L}(u, U_1, \dots, U_L); Q_\nu(x_0,\varepsilon))$ and therefore
we obtain, using also Remark \ref{A1A2}, (H4) and \eqref{infmvj},
\begin{align}
m(u,U_1,\dots, U_L;Q_\nu(x_0,\varepsilon))& \leq  
\mathcal F(\widetilde v_\varepsilon,\widetilde V^1_\varepsilon,\dots, \widetilde V^L_\varepsilon;Q_\nu(x_0,\varepsilon))\nonumber\\
&\leq \mathcal F(\widetilde v_\varepsilon, \widetilde V^1_\varepsilon,\dots, \widetilde V^L_\varepsilon;Q_\nu(x_0,\delta'\varepsilon))+ \mathcal F(\widetilde v_\varepsilon, \widetilde V^1_\varepsilon,\dots, \widetilde V^L_\varepsilon; Q_\nu(x_0,\varepsilon)\setminus \overline{ Q_\nu(x_0,\delta\varepsilon)})\nonumber \\
&\leq \mathcal F(\widetilde v_\varepsilon, \widetilde V^1_\varepsilon,\dots, \widetilde V^L_\varepsilon;Q_\nu(x_0,\delta\varepsilon))+ \mathcal F(\widetilde v_\varepsilon, \widetilde V^1_\varepsilon,\dots, \widetilde V^L_\varepsilon; Q_\nu(x_0,\varepsilon)\setminus \overline{ Q_\nu(x_0,\delta\varepsilon)})\nonumber \\
& \hspace{1cm} + C\Big(\mathcal L^N(Q_\nu(x_0,\delta'\varepsilon) \setminus Q_\nu(x_0,\delta\varepsilon)) + |D \widetilde v_\varepsilon|(Q_\nu(x_0,\delta'\varepsilon)\setminus Q_\nu(x_0,\delta \varepsilon))\nonumber \\
& \hspace{2cm} + \sum_{i=1}^L \int_{Q_\nu(x_0,\delta'\varepsilon)\setminus Q_\nu(x_0,\delta \varepsilon)} |\widetilde V^i_\varepsilon  |^p \,\de x\Big)\nonumber \\
&\leq \mathcal F(\widetilde u, \widetilde U_1,\dots, \widetilde U_L;Q_\nu(x_0,\delta\varepsilon)) + C\Big(\mathcal L^N(Q_\nu(x_0,\varepsilon) \setminus Q_\nu(x_0,\delta\varepsilon))\nonumber \\ 
& \hspace{2cm} + \sum_{i=1}^L \int_{Q_\nu(x_0,\varepsilon)\setminus Q_\nu(x_0,\delta \varepsilon)} |\widetilde V^i_\varepsilon  |^p \,\de x 
+ |D \widetilde v_\varepsilon|(Q_\nu(x_0,\varepsilon)\setminus Q_\nu(x_0,\delta \varepsilon))\Big)\nonumber \\
&\leq \varepsilon^{N}+ m(v_j,0,\dots,0; Q_\nu(x_0,\delta\varepsilon))\nonumber \\	
& \hspace{0.8cm}+ C\Big( \varepsilon^N (1-\delta^N) + 
|D u|(Q_\nu(x_0,\varepsilon)\setminus \overline{Q_\nu(x_0,\delta \varepsilon)})+ |{\rm tr}\,\widetilde u-{\rm \tr}\, u|(\partial Q_\nu(x_0,\delta\varepsilon))\nonumber \\
&\hspace{2cm}+\sum_{i=1}^L \int_{Q_\nu(x_0,\varepsilon)\setminus Q_\nu(x_0,\delta \varepsilon)} |\widetilde V^i_\varepsilon  |^p \,\de x\Big). \label{mestjump2}
\end{align}

For every $i \in \{1,\dots,L\}$ we have, using  H\"older's inequality,
\begin{align}
\int_{Q_\nu(x_0,\varepsilon)\setminus Q_\nu(x_0,\delta \varepsilon)} 
|\widetilde V^i_\varepsilon  |^p \, \de x &\leq 
\frac{1}{\varepsilon ^{N(p-1)}(1-\delta^N)^{p-1}}\left|\int_{Q_\nu(x_0,\varepsilon)} U_i(x) \, \de x \right|^p \nonumber\\
&\leq \frac{\varepsilon^{N(p-1)}}{\varepsilon ^{N(p-1)}(1-\delta^N)^{p-1}} \|U_i\|^p_{L^p(Q_\nu(x_0,\varepsilon);\mathbb R^{d\times N})} \nonumber\\
&= \frac{1}{(1-\delta^N)^{p-1}} 
\|U_i\|^p_{L^p(Q_\nu(x_0,\varepsilon);\mathbb R^{d\times N})}. \label{mest4}
\end{align}

Hence, from \eqref{mestjump2} and \eqref{mest4}, taking into account \eqref{Uip} and Lemma \ref{FM}, it follows that
\begin{align}
\lim_{\varepsilon \to 0^+}\frac{m(u,U_1,\dots U_L; Q_\nu(x_0, \varepsilon))}{\varepsilon^{N-1}} &\leq 
\limsup_{\delta \to 1^-}\limsup_{\varepsilon \to 0^+} 
C\Big(\varepsilon +\frac{m(v_j,0,\dots, 0;Q_\nu (x_0,\delta\varepsilon))}
{\varepsilon^{N-1}}\nonumber\\
&\hspace{1cm} + \varepsilon(1-\delta^N)+\frac{1}{(1-\delta^N)^{p-1}}
\sum_{i=1}^L\frac{1}{\eps^{N-1}} \|U_i\|^p_{L^p(Q_\nu(x_0,\varepsilon);\mathbb R^{d\times N})} \nonumber\\
&\hspace{1cm} +\frac{|D u|(Q_\nu(x_0,\varepsilon)\setminus \overline{Q_\nu(x_0,\delta \varepsilon)})}{\varepsilon^{N-1}} + \frac{|{\rm tr}\tilde u-{\rm \tr} u|(\partial Q_\nu(x_0,\delta\varepsilon))}{\varepsilon^{N-1}}\Big)\nonumber
\end{align}
\begin{align}
&\hspace{2cm}\leq \lim_{\varepsilon \to 0^+}\frac{m(v_j,0,\dots 0; Q_\nu(x_0, \varepsilon))}{\varepsilon^{N-1}} + \limsup_{\delta \to 1^-}(1-\delta^N) |[u]|(x_0)\nonumber\\
&\hspace{2cm}= \lim_{\varepsilon \to 0^+}\frac{m(v_j,0,\dots 0; Q_\nu(x_0, \varepsilon))}{\varepsilon^{N-1}}\label{ineq}
\end{align}
since, by \cite[(5.79)]{AFP} and \eqref{Du},
\begin{align*}
\lim_{\varepsilon \to 0^+}\frac{|Du|(Q_\nu(x_0,\varepsilon)\setminus \overline{Q_\nu(x_0,\delta\varepsilon)})}{\varepsilon^{N-1}}\leq (1-\delta^N)|[u]|(x_0).
\end{align*}
and 
\begin{equation}\label{traces}
\lim_{\varepsilon \to 0^+}\frac{|{\rm tr}\tilde u-{\rm \tr} u|(\partial Q_\nu(x_0,\delta\varepsilon))}{\varepsilon^{N-1}}=0.
\end{equation}
To prove this last fact we change variables and use \eqref{trvj} to obtain
\begin{align*}
\lim_{\varepsilon \to 0^+}\frac{|{\rm tr}\tilde u-{\rm \tr} u|(\partial Q_\nu(x_0,\delta\varepsilon))}{\varepsilon^{N-1}}
&=\lim_{\varepsilon \to 0^+}\delta^{N-1}\int_{\partial Q_\nu}
|{\rm tr}(v_j(x_0+\delta \eps y) - u(x_0+\delta \eps y))| \, \de {\mathcal H}^{N-1}(y)\\
&= \lim_{\varepsilon \to 0^+}\delta^{N-1}\int_{\partial Q_\nu}
|{\rm tr}(v_{u^+(x_0),u^-(x_0),\nu(x_0)}(y)-u_{\delta \varepsilon}(y))| 
\, \de {\mathcal H}^{N-1}(y).
\end{align*}
where
$u_{\delta\varepsilon}(y) = u(x_0+\delta \varepsilon y)$. Then \eqref{jumppt} and \eqref{Du} yield 
$$u_{ \delta\varepsilon}\to v_{u^+(x_0),u^-(x_0),\nu(x_0)} \mbox{ in } L^1(Q_\nu;\mathbb R^d) \mbox{ as } \eps \to 0^+$$
and
$$|D u_{\delta\varepsilon}|(Q_\nu)=\frac{1}{(\delta\varepsilon)^{N-1}}
|Du|(Q_\nu(x_0,\delta \varepsilon))\to |[u]|(x_0)=|Dv_{u^+(x_0),u^-(x_0),\nu(x_0)}|(Q_\nu)
\mbox{ as } \eps \to 0^+.$$
Hence \eqref{traces} follows from \cite[Lemma 2.3]{BFM} and this completes the proof of inequality \eqref{ineq}. The reverse inequality can be shown in a similar way by interchanging the roles of $(u,U_1,\dots, U_L)$ and $(v_j,0,\dots, 0)$
leading to the conclusion stated in \eqref{mestjump}.

Theorem \ref{GMthmHSD} is thus proved.
\end{proof}

\section{Applications}\label{appl}

In this section we present some applications of the global method for relaxation obtained in Theorem \ref{GMthmHSD}. 

\subsection{$2$-level (first-order) structured deformations}\label{2level}

The first application concerns the case of a two-level structured deformation, that is, we take $L=1$ in Definition \ref{Def2.1}.
In this setting, given a deformation $u\in SBV(\Omega;\mathbb R^d)$, and two non-negative functions $W\colon\Omega\times\R{d\times N}\to[0,+\infty)$ and $\psi\colon\Omega\times\R{d}\times\S{N-1}\to[0,+\infty)$, we consider the initial energy of $u$ defined by
\begin{equation}\label{103}
	E(u)\coloneqq \int_\Omega W(x,\nabla u(x))\,\de x+\int_{\Omega\cap S_u} \psi(x,[u](x),\nu_u(x))\,\de\cH^{N-1}(x),
\end{equation}
which is determined by the bulk and surface energy densities $W$ and $\psi$, respectively.

Then, as justified by the Approximation Theorem \ref{appTHMh}, we assign an energy to a structured deformation $(g,G)\in HSD^p_1(\Omega)$, which is equivalent to saying that $(g, G)\in SD(\Omega)$ and $G \in L^p(\Omega;\mathbb R^{d\times N})$, via
\begin{equation}\label{102}
	I_p(g,G)\coloneqq \inf\Big\{\liminf_{n\to\infty} E(u_n): u_n \in SBV(\Omega;\R{d}), u_n\wSD{*}(g,G)\Big\}.
\end{equation}

To simplify notation, here and in what follows, we write $u_n\wSD{*}(g,G)$ to mean 
$u_n \to g$ in $L^1(\Omega;\mathbb R^d)$ and $\nabla u_n \wto G$ in $L^p(\Omega;\mathbb R^{d\times N})$, if $p > 1$, and $\nabla u_n \wsto G$ in $\cM(\Omega;\mathbb R^{d\times N})$, if $p = 1$. Notice that this notion of convergence coincides, in the case $L=1$, with the one given in Definition \ref{S000}.

Under our coercivity hypothesis \eqref{W4} below, the definition of $I_p$ coincides with the one considered in \cite{CF1997}, see \cite[Remark 2.15]{CF1997}.

The functional in \eqref{102} was studied in \cite{CF1997}, in the homogeneous case, and later in \cite{MMOZ}, in the case of a uniformly continuous $x$ dependence, where, under certain hypotheses on $W$ and $\psi$ (cf. \cite[Theorem 5.1]{MMOZ}) it was shown that $I_p$ admits an integral representation, that is, that there exist functions $H_p\colon\Omega\times\R{d\times N}\times\R{d\times N}\to[0,+\infty)$ and $h_p\colon\Omega\times \R{d}\times\S{N-1}\to[0,+\infty)$ such that
\begin{equation}\label{104}
	I_p(g,G)=\int_\Omega H_p(x,\nabla g(x),G(x))\,\de x+
	\int_{\Omega\cap S_g} h_p(x, [g](x),\nu_g(x))\,\de\cH^{N-1}(x).
\end{equation}
In order to present the expressions of the relaxed energy densities $H_p$ and $h_p$ we start by  introducing some notation.

For $A,B\in\R{d\times N}$ let
\begin{equation}\label{T001}
\cC_p^{\bulk}(A,B)\coloneqq \bigg\{u\in SBV(Q;\R{d}): u|_{\partial Q}(x)=Ax, \int_Q \nabla u\,\de x=B, |\nabla u|\in L^p(Q) \bigg\}, 
\end{equation}
and for $\lambda\in\R{d}$ and $\nu\in\S{N-1}$ let
 $u_{\lambda,\nu}$ be the function defined by
\begin{equation}\label{909}
u_{\lambda,\nu}(x)\coloneqq 
\begin{cases}
\lambda & \text{if $x\cdot\nu\geq0$,} \\
0 & \text{if $x\cdot\nu<0$,}
\end{cases}
\end{equation}
and consider the classes given by
\begin{equation*}
\cC_p^\surface(\lambda,\nu)\coloneqq \Big\{u\in SBV(Q_\nu;\R{d}): u|_{\partial Q_\nu}(x)=u_{\lambda,\nu}(x), \nabla u(x)=0\;\text{for $\cL^N$-a.e.~$x\in Q_\nu$}\Big\},
\end{equation*}
for $p > 1$, and for $p = 1$,
\begin{equation*}
\cC_1^\surface(\lambda,\nu)\coloneqq \Big\{u\in SBV(Q_\nu;\R{d}): u|_{\partial Q_\nu}(x)=u_{\lambda,\nu}(x), \int_Q \nabla u\,\de x=0\Big\}.
\end{equation*}
Then, the functions $H_p$ and $h_p$ appearing in \eqref{104} are given by (cf. \cite[(5.6), (5.7)]{MMOZ})
\begin{equation}\label{906}
H_p(x_0,A,B)\coloneqq \inf\bigg\{  
\int_Q W(x_0,\nabla u(x))\,\de x+\int_{Q\cap S_u} \psi(x_0,[u](x),\nu_u(x))\,\de\cH^{N-1}(x): 
u\in\cC_p^\bulk(A,B)\bigg\},
\end{equation}
for all $x_0\in\Omega$ and $A,B\in\R{d\times N}$,
and, for all $x_0\in\Omega$, $\lambda\in\R{d}$ and $\nu\in\S{N-1}$,
\begin{equation}\label{907}
\!\!\! h_p(x_0,\lambda,\nu)\coloneqq \inf\bigg\{ 
\delta_1(p) \!\! \int_{Q_\nu} \!\!\!\! W^\infty(x_0,\nabla u(x))\,\de x+ \! \int_{Q_\nu\cap S_u} \!\!\!\!\!\!\!\!\! \psi(x_0,[u](x),\nu_u(x))\,\de\cH^{N-1}(x): 
u\in\cC_p^\surface(\lambda,\nu)\bigg\},
\end{equation}
where $W^{\infty}$ denotes the \textit{recession function} at infinity of $W$ with respect to the second variable, given by
\begin{equation*}
W^{\infty}(x,A) \coloneqq \limsup_{t\rightarrow +\infty} 
\frac {W(x,tA)}{t}, \; \forall x \in \Omega, \forall A\in\R{d\times N}.
\end{equation*}
In \eqref{907}, $\delta_1(p) = 1$ if $p=1$ and $\delta_1(p) = 0$ if $p\neq 1$, so that the relaxed surface energy density depends on the recession function of $W$ only in the case $p=1$.

In what follows we obtain an integral representation result for $I_p(g,G)$, by means of Theorem \ref{GMthmHSD}, under a similar set of hypotheses on $W$ and $\psi$ as those considered in \cite{CF1997} and \cite{MMOZ}, 
but requiring only measurability, rather than uniform continuity, of $W$ in the $x$ variable. 

Precisely, we assume that $W\colon\Omega \times \mathbb R^{d \times N}\to[0,+\infty)$ and $\psi\colon \Omega \times \mathbb R^d \times \mathbb S^{N-1} \to [0,+\infty) $ are Carath\'eodory functions such that the following conditions hold: 
\begin{enumerate}
\item \label{(W1)_p}  ($p$-Lipschitz continuity) there exists $C_W >0$ such that, for a.e. $x\in\Omega$ and $A_1,A_2 \in \R{d\times N}$,
	\begin{equation*}
		|W(x,A_1) - W(x,A_2)| \leq C_W |A_1 - A_2| \big(1+|A_1|^{p-1}+|A_2|^{p-1}\big);
	\end{equation*}
\item \label{W3} there exists $A_0 \in \mathbb R^{d \times N}$ such that 
$W(\cdot, A_0)\in L^\infty(\Omega)$;
\item \label{W4} there exists $c_W>0$ such that, for a.e. $x \in \Omega$ and every $A \in  \mathbb R^{d\times N}$,
	\begin{equation*}
		c_W |A|^{p}-\frac{1}{c_W}\leq W(x,A);
	\end{equation*}
\item\label{psi_0} (symmetry) for every $x \in \Omega$, $\lambda \in \R{d}$ and 
$\nu \in \mathbb S^{N-1}$, 
	\begin{equation*}
		\psi (x, \lambda, \nu)= \psi (x,-\lambda, -\nu);
	\end{equation*}
\item\label{(psi1)} there exist $c_\psi,C_\psi > 0$ such that, for all $x\in\Omega$, $\lambda \in \R{d}$ and $\nu \in \S{N-1}$,
	\begin{equation*}
		c_\psi|\lambda| \leq \psi(x,\lambda, \nu) \leq C_\psi|\lambda |;
	\end{equation*}
\item\label{(psi2)} (positive $1$-homogeneity) for all $x\in\Omega$, $\lambda \in \R{d}$, $\nu \in \S{N-1}$ and $t >0$
		$$\psi(x,t\lambda, \nu) = t\psi(x, \lambda, \nu);$$
\item \label{(psi3)} (sub-additivity) for all $x\in\Omega$, $\lambda_1,\lambda_2 \in \R{d}$ and $\nu \in \S{N-1}$,
		\begin{equation*}
			\psi(x, \lambda_1 + \lambda_2, \nu) \leq \psi(x,\lambda_1, \nu) +\psi(x,\lambda_2, \nu);
		\end{equation*}
\item \label{(psi4)} there exists a continuous function $\omega_\psi\colon[0,+\infty)\to[0,+\infty)$ with $\omega_\psi(s)\to 0$ as $s\to0^+$ such that, for every $x_0,x_1\in\Omega$, $\lambda \in \R{d}$ and $\nu \in \S{N-1}$,
		\begin{equation*}
	|\psi(x_1,\lambda,\nu)-\psi(x_0,\lambda,\nu)|\leq\omega_\psi(|x_1-x_0|)|\lambda|.
		\end{equation*}
\end{enumerate}

\medskip

Under this set of hypotheses we prove the following theorem.

\begin{theorem}\label{representation}
	Let $p\geq 1$ and let $\Omega \subset \mathbb R^N$ be a bounded, open set. 
	Consider $E$ given by \eqref{103} where  $W\colon\Omega\times\R{d\times N}\to[0,+\infty)$ and $\psi\colon\Omega\times\R{d}\times\S{N-1}\to[0,+\infty)$ satisfy \eqref{(W1)_p}-\eqref{(psi1)} and $\psi$ is continuous. Let $(g,G) \in SD(\Omega)$, with 
	$G \in L^p(\Omega;\mathbb R^{d \times N})$, 
	and assume that $I_p(g,G)$ is defined by \eqref{102}. 
	
	Then,  there exist 
	$f: \Omega \times \mathbb R^{d \times N} \times \mathbb R^{d \times N} \to [0,+\infty)$, $\Phi: \Omega \times \mathbb R^d \times S^{N-1}\to [0, +\infty)$ such that  
	\begin{align}\label{reprelax}
		I_p(g,G) =  \int_\Omega f(x,\nabla g(x),G(x)) \, \de x + 
		\int_{\Omega\cap S_g}\Phi(x,[g](x),\nu_g(x)) \, \de \mathcal H^{N-1}(x),
		\end{align}
	 where the relaxed energy densities are given by
	\begin{align}\label{fdef}
		f(x_0, \xi, B):=	\limsup_{\varepsilon \to 0^+}\frac{m(\xi(\cdot-x_0),B; Q(x_0,\varepsilon))}{\varepsilon^N},
	\end{align}
	\begin{align}\label{Phidef}
		\Phi(x_0,\lambda,\theta,\nu):= \limsup_{\varepsilon \to 0^+}\frac{m(u_{\lambda-\theta,\nu}(\cdot-x_0), 0; Q_{\nu}(x_0,\varepsilon))}{\varepsilon ^{N-1}},
	\end{align}
	for all $x_0\in \Omega$, $ \theta,\lambda \in \mathbb R^d$, 
	$\xi, B \in \mathbb R^{d \times N}$ and $\nu \in \mathbb S^{N-1}$. In the above expressions $0$ denotes the zero $\mathbb R^{d \times N}$ matrix,  
	$u_{\lambda-\theta, \nu}(y) := \begin{cases} 
	\lambda - \theta, &\hbox{if } y\cdot \nu > 0\\
	0, &\hbox{ if } y\cdot \nu \leq 0
	\end{cases}$,
	the functional $m\colon SD(\Omega)
	\times \mathcal O_\infty(\Omega)\to [0,+\infty)$ is given by \eqref{mdef} with $L=1$ and $\mathcal F= I_p$, and $\mathcal C_{HSD^p_1}(g,G;O)$ is given by \eqref{classM}, taking into account that $HSD^p_1(\Omega)$ in Definition \ref{Def2.1} coincides with the set of fields $(g, G)\in SD(\Omega)$ such that 
	$G \in L^p(\Omega;\mathbb R^{d\times N})$.
	
	Furthermore, if $p>1$ and $\psi$ also satisfies \eqref{(psi2)}-\eqref{(psi4)},  then $\Phi(x_0,\lambda,\theta,\nu) = h_p(x_0,\lambda-\theta,\nu)$,
	for every $x_0\in \Omega$, $\theta,\lambda \in \mathbb R^d$ and 
	$\nu \in \mathbb S^{N-1}$, where $h_p$ is the function given in \eqref{907}.
\end{theorem}

\begin{proof}[Proof]
	
Given $O \in \mathcal O(\Omega)$ and $(g,G) \in SD(\Omega)$, with 
$G \in L^p(\Omega;\mathbb R^{d \times N})$, we introduce the localized version of 
$I_p(g, G)$, namely
\begin{align*}
I_p(g,G;O)\coloneqq \inf\Big\{\liminf_{n\to\infty} E(u_n): u_n \in SBV(O;\R{d}), u_n\wSD{*}(g,G) \hbox{ in }O\Big\}.
\end{align*}

Our goal is to verify that $I_p(g,G;O)$ satisfies assumptions (H1)-(H4) of 
Theorem \ref{GMthmHSD} in the case $L=1$.

We start by proving the following nested subadditivity result: if $O_1, O_2, O_3$ are open subsets of $\Omega$ such that 
$O_1 \Subset O_2 \subseteq O_3$, then
$$I_p(g,G;O_3) \leq I_p(g,G;O_2) + I_p(g,G;O_3\setminus \overline {O_1}).$$
Indeed, let $u_n \in SBV(O_2;{\mathbb{R}}^d)$ and 
$v_n \in SBV(O_3\setminus \overline{O_1};{\mathbb{R}}^d)$ be two sequences such that $u_n \rightarrow g$ in $L^1(O_2;{\mathbb{R}}^d)$, 
$\nabla u_n \rightharpoonup G$ in $L^p(O_2;{\mathbb{R}}^{d\times N})$, 
$v_n \rightarrow g$ in $L^1(O_3\setminus \overline{O_1};{\mathbb{R}}^d)$, 
$\nabla v_n \rightharpoonup G$ in 
$L^p(O_3\setminus \overline{O_1};{\mathbb{R}}^{d\times N})$, and, in addition, 
\begin{equation*}
I_p(g,G;O_2) = \lim_{n\to +\infty}\left[\int_{O_2}W(x, \nabla u_n(x)) \, \de x + \int_{S_{u_n}\cap O_2}
\psi(x,[u_n](x),\nu_{u_n}(x)) \, \de\mathcal{H}^{N-1}(x)\right] 
\end{equation*}
\noindent and 
\begin{equation*}
I_p(g,G;O_3\setminus \overline{O_1}) = \lim_{n\to +\infty} 
\left[\int_{O_3\setminus \overline{O_1}}W(x, \nabla v_n(x)) \, \de x + \int_{S_{v_n}\cap(O_3\setminus \overline{O_1})} \psi(x,[v_n](x),\nu_{v_n}(x))\, \de\mathcal{H}^{N-1}(x)\right].
\end{equation*}
\noindent Notice that 
\begin{equation}  \label{L1conv}
u_n - v_n \rightarrow 0 \; \text{ in }\; 
L^1(O_2 \cap (O_3\setminus \overline{O_1});{\mathbb{R}}^d)
\end{equation}
and 
\begin{equation*}
\nabla u_n - \nabla v_n \rightharpoonup 0 \; \text{ in }\; 
L^p(O_2 \cap(O_3\setminus \overline{O_1});{\mathbb{R}}^{d\times N}).
\end{equation*}
\noindent For $\delta > 0$ define 
\begin{equation*}
O_{\delta} := \{ x \in O_2: \,\, \mbox{dist}(x, O_1) < \delta\}.
\end{equation*}
\noindent For $x \in O_3$ let $d(x):= \mbox{dist}(x, O_1)$. Since the distance
function to a fixed set is Lipschitz continuous (see \cite[Exercise 1.1]{Z}), we can apply the change of variables formula \cite[Section 3.4.3, Theorem 2]{EG}, to obtain 
\begin{equation*}
\int_{O_{\delta}\setminus \overline{O_1}} |u_n(x) - v_n(x)| \, |\det\nabla d(x)| \, \de x =
\int_{0}^{\delta}\left [ \int_{d^{-1}(y)} |u_n(x) - v_n(x)| \, 
\de \mathcal{H}^{N-1}(x)\right]\, \de y
\end{equation*}
\noindent and, as $|\det\nabla d|$ is bounded and \eqref{L1conv} holds,
 by Fatou's Lemma, it
follows that for almost every $\rho \in [0, \delta]$ we have 
\begin{equation}  \label{aer}
\liminf_{n\rightarrow +\infty} 
\int_{d^{-1}(\rho)} |u_n(x) - v_n(x) |\, \de\mathcal{H}^{N-1}(x) 
= \liminf_{n\rightarrow +\infty} \int_{\partial O_{\rho}}
|u_n(x) - v_n(x) |\, \de\mathcal{H}^{N-1}(x) = 0.
\end{equation}
Fix $\rho_0 \in [0, \delta]$ such that 
$\|G \chi_{O_2}\|(\partial O_{\rho_0}) = 0$, 
$\|G \chi_{O_3 \setminus \overline{O_1}}\|(\partial O_{\rho_0}) = 0$ and
such that \eqref{aer} holds. For this choice of $\rho_0$, we may pass to subsequences of  $u_n$ and $v_n$ (not relabelled) such that the liminf in \eqref{aer} is actually a limit. \color{black} We observe that $O_{\rho_0}$ is a set with
locally Lipschitz boundary since it is a level set of a Lipschitz function
(see, e.g., \cite{EG}). Hence we can consider $u_n, v_n$ 
on $\partial O_{ \rho_0}$ in the sense of traces and define 
\begin{equation*}
w_n = 
\begin{cases}
u_n & \text{ in}\; \overline{O}_{\rho_0} \\ 
v_n & \text{ in}\; O_3\setminus \overline{O}_{\rho_0}.
\end{cases}
\end{equation*}
\noindent Then, by the choice of $\rho_0$, $w_n$ is admissible for 
$I_p(g,G;O_3)$ so, by \eqref{(psi1)}, \eqref{L1conv} and \eqref{aer}, we obtain 
\begin{align*}
I_p(g,G;O_3) &\leq  \liminf_{n\to +\infty} 
\left[\int_{O_3}W(x, \nabla w_n(x)) \, \de x +
\int_{S_{w_{n}}\cap O_3} \psi(x,[w_n](x),\nu_{w_n}(x))\, \de\cH^{N-1}(x)\right] \\
&\leq  \liminf_{n\to +\infty} 
\left[\int_{O_2}W(x, \nabla u_n(x)) \, \de x +
\int_{S_{u_{n}}\cap O_2} \psi(x,[u_n](x),\nu_{u_n}(x))\, \de\cH^{N-1}(x) \right. \\
&  \hspace{1cm} + \int_{O_3\setminus \overline{O_1}}W(x, \nabla v_n(x)) \, \de x  + \int_{S_{v_{n}}\cap (O_3 \setminus \overline{O_1})}
\psi(x,[v_n](x),\nu_{v_n}(x))\, \de\cH^{N-1}(x) \\
&  \left. \hspace{1cm} + \int_{S_{w_n} \cap \partial O_{\rho_0}} 
C |u_n(x) - v_n(x)| \, \de\mathcal{H}^{N-1}(x) \right] \\
& =  I_p(g,G;O_2) + I_p(g,G;O_3 \setminus \overline{O_1}),
\end{align*}
\noindent which concludes the proof.

From here, the reasoning in 
\cite[ Proposition 2.22]{CF1997}, 	
which is still valid with the same proof in the non-homogeneous case, yields (H1).		
		
To show that (H2) holds, we argue as in \cite[Proposition 5.1]{CF1997}. Indeed, we can prove lower semicontinuity of $I_p(\cdot,\cdot;O)$ along sequences $(g_n,G_n)$ converging in 
$L^1(O;\mathbb R^d)_{\it strong}\times L^p(O;\mathbb R^{d\times N})_{\it weak}$
(the second convergence is weak star in $\cM(\Omega;\R{d\times N})$, if $p=1$).
	
(H3) is an immediate consequence of the previous lower semicontinuity property in $O$, as observed in \cite[eq. (2.2)]{BFM}, whereas
(H4) follows by standard arguments (as in \cite[Lemma 2.18]{CF1997}) from (1), (2), (3) and (6) above and by the lower semicontinuity of integral functionals of power type and the total variation along weakly converging sequences. We point out that in order to obtain the lower bound in (H4) we can replace, without loss of generality, $W$ by $W + \frac{1}{C_W}$.
	
Hence, Theorem \ref{GMthmHSD} can be applied to conclude that,
for every $(g,G)\in SD(\Omega) \times L^p(\Omega;\mathbb R^{d \times N})$, 
we have
$$I_p(g,G) =  \int_\Omega f(x, g(x),\nabla g(x), G(x)) \, \de x + 
\int_{\Omega\cap S_g}\Phi(x, g^+(x), g^-(x),\nu_g(x)) \,\de \mathcal H^{N-1}(x),$$
where the relaxed densities $f$ and $\Phi$ are given by 
$$
f(x_0, a, \xi, B) =	\limsup_{\varepsilon \to 0^+}\frac{m(a+ \xi(\cdot-x_0),B; Q(x_0,\varepsilon))}{\varepsilon^N},$$
and
$$\Phi(x_0, \lambda, \theta, \nu)= \limsup_{\varepsilon \to 0^+}\frac{m(v_{\lambda, \theta,\nu}(\cdot-x_0), 0; Q_{\nu}(x_0,\varepsilon))}{\varepsilon ^{N-1}},$$
for all $x_0\in \Omega$, $ a, \theta,\lambda \in \mathbb R^d$, 
$\xi, B \in \mathbb R^{d \times N}$, $\nu \in \mathbb S^{N-1}$.

It is easy to see that the functional $I_p$ is invariant under translation in the first variable, that is,  
$$I_p(g+a,G;O)= I_p(g,G;O), \; \forall (g,G) \in SD(\Omega), O \in \mathcal O(\Omega), a \in \mathbb R^d.$$ 
Indeed, it suffices to notice that if $\{u_n\}$ is admissible for $I_p(g,G;O)$, then the sequence $u_n + a$ is admissible for $I_p(g+a,G;O)$.
Hence, taking into account Remark \ref{traslinv} and the abuse of notation stated therein, we obtain \eqref{reprelax} with $f$ and $\Phi$ given by 
\eqref{fdef} and \eqref{Phidef}, respectively.

On the other hand, for $p>1$, Theorem \ref{thm4.3FHP} and the fact that $\mathcal F= I_p$, yield, 
for every $x_0\in \Omega$, $\lambda, \theta \in \mathbb R^{d}$ and
$\nu \in \mathbb S^{N-1}$, 
\begin{align*}
\Phi(x_0,\lambda,\theta,\nu) &= \Phi(x_0,\lambda -\theta,\nu) = 
\limsup_{\varepsilon \to 0^+}\frac{m(u_{\lambda-\theta,\nu}(\cdot-x_0), 0; Q_{\nu}(x_0,\varepsilon))}{\varepsilon ^{N-1}} \\
&=\limsup_{\varepsilon \to 0^+}\frac{I_p(u_{\lambda-\theta,\nu}(\cdot-x_0), 0; Q_{\nu}(x_0,\varepsilon))}{\varepsilon ^{N-1}}\\
&= \limsup_{\varepsilon \to 0^+}\frac{1}{\varepsilon^{N-1}}
\inf\Bigg\{\liminf_{n\to+\infty}\left[\int_{Q_\nu(x_0,\varepsilon)} 
\hspace{-0,85cm}W(x,\nabla u_n(x)) \, \de x + \int_{Q_\nu(x_0,\varepsilon)\cap S_{u_n}}
\hspace{-0,81cm}\psi(x, [u_n](x),\nu_{u_n}(x)) \,\de\mathcal H^{N-1}(x)\right] :\\
\\
&\hspace{3,7cm} u_n \in SBV(Q_\nu(x_0,\varepsilon);\mathbb R^d), 
u_n \to u_{\lambda-\theta,\nu}(\cdot-x_0) \mbox{ in }  L^1(Q_\nu(x_0,\varepsilon);\mathbb R^d),\\
&\hspace{8,7cm} \nabla u_n \rightharpoonup 0 
\mbox{ in } L^p(Q_\nu(x_0,\varepsilon);\mathbb R^{d\times N}) \Bigg\} \\
&\leq \limsup_{\varepsilon \to 0^+}\frac{1}{\varepsilon^{N-1}}
\inf\Bigg\{\liminf_{n\to+\infty}\int_{Q_\nu(x_0,\varepsilon)\cap S_{u_n}}
\hspace{-0,78cm}\psi(x, [u_n](x),\nu_{u_n}(x)) \,\de\mathcal H^{N-1}(x) :\\
\\
&\hspace{3,7cm} u_n \in SBV(Q_\nu(x_0,\varepsilon);\mathbb R^d), 
u_n \to u_{\lambda-\theta,\nu}(\cdot-x_0) \mbox{ in }  L^1(Q_\nu(x_0,\varepsilon);\mathbb R^d),\\
&\hspace{8,7cm} \nabla u_n  = 0 
\hbox{ a.e. in } Q_\nu(x_0,\varepsilon)\Bigg\},
\end{align*}
where $u_{\lambda-\theta, \nu}$ is given by \eqref{909}, with $\lambda$ replaced by $\lambda - \theta$, and we have taken into account the growth condition on $W$ given by (1) and hypothesis (2), and the fact that the latter class of test functions is contained in the initial one.

Given that this last expression no longer depends on the initial bulk density $W$, but only on $\psi$ for which the uniform continuity condition \eqref{(psi4)} holds, we may apply this condition to replace $x$ by $x_0$ and obtain
\begin{align*}
\Phi(x_0,\lambda,\theta,\nu) &
\leq \limsup_{\varepsilon \to 0^+}\frac{1}{\varepsilon^{N-1}}
\inf\Bigg\{\liminf_{n\to+\infty}\int_{Q_\nu(x_0,\varepsilon)\cap S_{u_n}}
\hspace{-0,78cm}\psi(x_0, [u_n](x),\nu_{u_n}(x)) \,\de\mathcal H^{N-1}(x) :\\
\\
&\hspace{3,7cm} u_n \in SBV(Q_\nu(x_0,\varepsilon);\mathbb R^d), 
u_n \to u_{\lambda-\theta,\nu}(\cdot-x_0) \mbox{ in }  L^1(Q_\nu(x_0,\varepsilon);\mathbb R^d),\\
&\hspace{8,7cm} \nabla u_n  = 0 
\hbox{ a.e. in } Q_\nu(x_0,\varepsilon)\Bigg\}.
\end{align*}
We now invoke the periodicity argument used in the first part of the proof of \cite[Proposition 4.2]{CF1997} to conclude that
\begin{align*}
\Phi(x_0,\lambda,\theta,\nu) & 
\leq\limsup_{\varepsilon \to 0^+}\frac{1}{\varepsilon^{N-1}}
\inf\Bigg\{\int_{Q_\nu(x_0,\varepsilon)\cap S_v}
\psi(x_0, [v](x),\nu_v(x)) \,\de\mathcal H^{N-1}(x): 
v \in SBV(Q_\nu(x_0,\varepsilon);\mathbb R^d),\\
&\hspace{4,2cm}
\nabla v(x) = 0 \hbox{ a.e. in }  Q_\nu(x_0,\varepsilon), 
v(\cdot - x_0)|_{\partial Q_\nu(x_0,\varepsilon)} = u_{\lambda-\theta,\nu}(\cdot - x_0) \Bigg\},
\end{align*}
which, by a simple change of variables, coincides with
\begin{align*}
&\inf\Bigg\{\int_{Q_\nu\cap S_v}
\psi(x_0, [u](y),\nu_u(y)) \,\de\mathcal H^{N-1}(y): 
u \in SBV(Q_\nu;\mathbb R^d),
\nabla u(x)= 0 \hbox{ a. e. in } Q_\nu,  
u|_{\partial Q_\nu} = u_{\lambda-\theta,\nu} \Bigg\}\\
&\hspace{1cm}= h_p(x_0, \lambda - \theta, \nu),
\end{align*} 
so it follows that
$$\Phi(x_0,\lambda,\theta,\nu) \leq  h_p(x_0, \lambda - \theta, \nu).$$

To prove the reverse inequality, we use the fact that $W\geq 0$ to obtain
\begin{align*}
	\Phi(x_0,\lambda,\theta,\nu) 
	&=\limsup_{\varepsilon \to 0^+}\frac{I_p(u_{\lambda-\theta,\nu}(\cdot-x_0), 0; Q_{\nu}(x_0,\varepsilon))}{\varepsilon ^{N-1}}\\
	&\geq \limsup_{\varepsilon \to 0^+}\frac{1}{\varepsilon^{N-1}}
	\inf\Bigg\{\liminf_{n\to+\infty}\int_{Q_\nu(x_0,\varepsilon)\cap S_{u_n}}
	\hspace{-0,78cm}\psi(x, [u_n](x),\nu_{u_n}(x)) \,\de\mathcal H^{N-1}(x) :\\
	\\
	&\hspace{2,7cm} u_n \in SBV(Q_\nu(x_0,\varepsilon);\mathbb R^d), 
	u_n \to u_{\lambda-\theta,\nu}(\cdot-x_0) \mbox{ in }  L^1(Q_\nu(x_0,\varepsilon);\mathbb R^d),\\
	&\hspace{8,7cm} 
	\nabla u_n \rightharpoonup 0\mbox{ in } L^p(Q_\nu(x_0,\varepsilon);\mathbb R^{d\times N}) 
	\Bigg\}
	\end{align*}
	\begin{align*}
	& \hspace{1cm}= \limsup_{\varepsilon \to 0^+}\frac{1}{\varepsilon^{N-1}}
		\inf\Bigg\{\liminf_{n\to+\infty}\int_{Q_\nu(x_0,\varepsilon)\cap S_{u_n}}
		\hspace{-0,78cm}\psi(x_0, [u_n](x),\nu_{u_n}(x)) \,\de\mathcal H^{N-1}(x) :\\
		\\
		&\hspace{3,7cm} u_n \in SBV(Q_\nu(x_0,\varepsilon);\mathbb R^d), 
		u_n \to u_{\lambda-\theta,\nu}(\cdot-x_0) \mbox{ in }  L^1(Q_\nu(x_0,\varepsilon);\mathbb R^d),\\
		&\hspace{8,7cm} 
		\nabla u_n \rightharpoonup 0\mbox{ in } L^p(Q_\nu(x_0,\varepsilon);\mathbb R^{d\times N}) 
		\Bigg\},
\end{align*}
where the uniform continuity of $\psi$ in the first variable was used in the final equality.

We now argue as in \cite[Propositions 4.2 and 4.4]{CF1997}, in order to replace  each weakly converging sequence $u_n$ by one which converges strongly to $0$ in $L^p$. In this way, we are lead to the conclusion that 
$$\Phi(x_0,\lambda, \theta, \nu) \geq h_p(x_0, \lambda-\theta, \nu).$$
We have thus proved that, for $p > 1$,
$$\Phi(x_0,\lambda,\theta,\nu) =  h_p(x_0, \lambda - \theta, \nu),$$
for every $x_0\in \Omega$, $\lambda, \theta \in \mathbb R^d$ and 
$\nu \in \mathbb S^{N-1}$, where $h_p$ is the function given in \eqref{907}.
 
\end{proof}

\begin{theorem}\label{contrepresentation}	
Let $p \geq 1$. Under the conditions of the previous theorem, if, in addition to the hypotheses stated therein, the density $W$ also satisfies
\begin{enumerate}
\item[(9)]\label{(9)}  there exists a continuous function $\omega_W:[0,+\infty)\to[0,+\infty)$ such that 
$\displaystyle \lim_{t \to 0^+} \omega_W(t)=0$ and
$$|W(x_1,A)-W(x_0,A)|\leq\omega_W(|x_1-x_0|)(1+|A|^{p}), \; 
\forall x_0, x_1 \in \Omega, A \in \mathbb R^{d\times N};$$
\item[(10)] \label{Winfty} if $p=1$, there exist $\alpha \in (0,1)$ and $L>0$ such that 	
\begin{equation*}
\bigg |W^{\infty}(x,A) - \frac{W(x,tA)}{t}\bigg| \leq \frac{C}{t^{\alpha}}, 	\end{equation*}
for all  $t>L$, $x \in \Omega$ and $A \in \mathbb R^{d \times N}$ with $|A|=1,$
\end{enumerate}
then 
\eqref{reprelax} holds for every $(g,G) \in SD(\Omega)$ such that
$G \in L^p(\Omega;\mathbb R^{d\times N})$,
with
$f(x_0,\xi, B) = H_p(x_0, \xi, B)$, for every $x_0\in \Omega$, 
$\xi, B \in \mathbb R^{d \times N}$, where $H_p$ is given by \eqref{906},
and $\Phi(x_0,\lambda,\theta,\nu) =  h_p(x_0, \lambda - \theta, \nu),$
for every $x_0\in \Omega$, $\lambda, \theta \in \mathbb R^d$ and 
$\nu \in \mathbb S^{N-1}$, where $h_p$ is the function given in \eqref{907}.
\end{theorem} 

\begin{proof}[Proof]
As seen in the previous proof, $I_p(g,G;\cdot)$ is the restriction to 
$\mathcal O(\Omega)$ of a Radon measure.
Assuming (9),  
putting together Theorem \ref{thm4.3FHP}, \eqref{fdef}
and \cite[Theorem 5.1]{MMOZ},
it follows that for every $x_0\in \Omega$, $\xi,B \in \mathbb R^{d\times N}$,
we have	
$$f(x_0,\xi, B) =H_p(x_0, \xi,B).$$
On the other hand, by (9), (10), \ref{thm4.3FHP}, \eqref{Phidef} and 
\cite[Theorem 5.1]{MMOZ}, we conclude that
$$\Phi(x_0,\lambda,\theta,\nu) =  h_p(x_0, \lambda - \theta, \nu),$$
for every $x_0\in \Omega$, $\lambda, \theta \in \mathbb R^d$ and 
$\nu \in \mathbb S^{N-1}$, where $h_p$ is the function given in \eqref{907}.
\end{proof}

\subsection{Homogenization problems}\label{hom}

Our method also applies to homogenization problems like the one considered in \cite{AMMZ}. Indeed, it improves the result therein allowing us to consider also the linear growth case, as well as Carathéodory, rather than continuous, bulk energy densities, as the following result states.

\begin{theorem}\label{homogenization}
	Let $p \geq 1$ and $\Omega \subset \mathbb R^N$ be a bounded, open set.  Let $\varepsilon \to 0^+$ and
	consider $E_\varepsilon$ given by 
	\begin{equation*}
		E_\varepsilon(u)\coloneqq \int_\Omega W(x/\varepsilon,\nabla u(x))\,\de x+\int_{\Omega\cap S_u} \psi(x/\varepsilon,[u](x),\nu_u(x))\,\de\cH^{N-1}(x),
	\end{equation*}
	where  $W\colon\Omega\times\R{d\times N}\to[0,+\infty)$ is a Carathéodory function and $\psi\colon\Omega\times\R{d}\times\S{N-1}\to[0,+\infty)$ is a continuous function, both 
	$Q$-periodic in the first variable and such that they satisfy \eqref{(W1)_p}-\eqref{(psi1)}. Let $(g,G) \in SD(\Omega)$
	and let $I_{p,hom}$ be the functional defined by
	\begin{equation}\label{102hom}
		I_{p,hom}(g,G)\coloneqq \inf\Big\{\liminf_{\varepsilon\to 0 } E_\varepsilon(u_\varepsilon): u_\varepsilon \in SBV(\Omega;\R{d}), u_\varepsilon\wSD{*}(g,G)\Big\}.
	\end{equation}
	
	Then, there exist 
	$f_{hom}: \mathbb R^{d \times N} \times \mathbb R^{d \times N} \to [0,+\infty)$, $\Phi_{hom}: \mathbb R^d \times S^{N-1}\to [0, +\infty)$ such that  
	\begin{align*}
		I_{p,hom}(g,G) =  \int_\Omega f_{hom}(\nabla g(x),G(x)) \, \de x + 
		\int_{\Omega\cap S_g}\Phi_{hom}([g](x),\nu_g(x)) \, \de \mathcal H^{N-1}(x),
	\end{align*}
	where the limiting energy densities are given by
	\begin{align}\label{fhomdef}
		f_{hom}(x_0,\xi, B):=	\limsup_{\varepsilon \to 0^+}\frac{m(\xi(\cdot-x_0),B; Q(x_0,\varepsilon))}{\varepsilon^N},
	\end{align}
	\begin{align}\label{Phihomdef}
		\Phi_{hom}(x_0,\lambda,\theta,\nu):= \limsup_{\varepsilon \to 0^+}\frac{m(v_{\lambda, \theta,\nu}(\cdot-x_0), 0; Q_{\nu}(x_0,\varepsilon))}{\varepsilon ^{N-1}},
	\end{align}
	for all $x_0\in \Omega$, $\lambda,  \theta \in \mathbb R^d$, 
	$\xi, B \in \mathbb R^{d \times N}$ and $\nu \in \mathbb S^{N-1}$. In the above expressions $0$ denotes the zero $\mathbb R^{d \times N}$ matrix,  
	$v_{\lambda,\theta, \nu}(y) := \begin{cases} 
		\lambda, &\hbox{if } y\cdot \nu > 0\\
		\theta, &\hbox{ if } y\cdot \nu \leq 0
	\end{cases}$,
	the functional $m\colon SD(\Omega)
	\times \mathcal O_\infty(\Omega)\to [0,+\infty)$ is given by \eqref{mdef} with $L=1$ and $\mathcal F= I_{p,hom}$, and $\mathcal C_{HSD^p_1}(g,G;O)$ is given by \eqref{classM}, taking into account that, if $p >1$, $HSD^p_1(\Omega)$ in Definition \ref{Def2.1} coincides with the set of fields $(g, G)\in SD(\Omega)$ such that 
	$G \in L^p(\Omega;\mathbb R^{d\times N})$.
	\end{theorem}

As in the case of Theorem \ref{representation}, the proof of this theorem amounts to the verification that the functional $I_{p,hom}$ satisfies all of the assumptions of Theorem \ref{GMthmHSD}, we omit the details. We also refer to \cite[Lemma A.1 and Proposition A.2]{AMMZ}, where the arguments were presented in the case $p>1$, but they can be repeated word for word if $p=1$. 
We point out that $f_{hom}$ is actually independent of $x_0$ and $a$ and $\Phi_{hom}$ is independent of $x_0$, due to the fact that $I_{p,hom}$ verifies the conditions of \cite[Lemma 4.3.3]{BFM} which in turn can be proven in full analogy with \cite[Lemma 3.7]{BDV}. 	  

Notice that, in view of the results in \cite{AMMZ}, in the case $p >1$ and assuming \eqref{(psi2)}-(9),
the densities given by \eqref{fhomdef} and \eqref{Phihomdef} coincide with the bulk and surface energy densities $H_{hom}$ and $h_{hom}$ obtained in \cite[eq. (1.11) and (1.12), respectively]{AMMZ}. In particular, when $p>1$ 
\eqref{Phihomdef} admits the equivalent representation (see \cite[Proposition 3.5]{AMMZ}),
	\begin{equation*}
		\begin{split}
			\!\! \Phi_{\hom}(\lambda,\theta,\nu)\coloneqq \limsup_{T\to+\infty}
			\frac1{T^{N-1}}
			\inf\bigg\{ & \int_{(TQ_\nu)\cap S_u}
			\!\!\!\! \!\!\!\! \psi(x,[u](x),\nu_u(x))\,\de\cH^{N-1}(x): 
			u\in SBV(T Q_\nu;\mathbb R^d), \\ 
			&\, u|_{\partial(TQ_\nu)}(x)=v_{\lambda,\theta, \nu}(x),\, \nabla u(x) =0 \hbox{ a.e. in } TQ_\nu\bigg\}
		\end{split}
	\end{equation*}
	for every $(\lambda,\theta, \nu)\in\mathbb R^d\times\mathbb R^d \times\mathbb S^{N-1}$.

\medskip

\subsection{Functionals arising in the analysis of second-order structured deformations}\label{SOSD}
	
As a further application of our abstract global method for relaxation, we recover the integral representation for one of the energies appearing in \cite{BMMO2017}.
In this paper a model for second-order structured deformations is proposed in the space $SBV_2$, giving rise to two energies (see \cite[Theorem 3.2]{BMMO2017}).
This decomposition relies strictly on hypotheses (I) and (II) below and \eqref{psi_0}-\eqref{(psi1)} from  Theorem \ref{representation}, but in the matrix setting.
Although the first of these energies 
does not satisfy the conditions of Theorem \ref{GMthmHSD},
we will apply our result 
to the second (to avoid confusion with the notation used in the previous applications we denote it here by $J$) which is defined on matrix-valued structured deformations, $J\colon SD(\Omega;\mathbb R^{d\times N}) \to [0,+\infty)$,
and is given by
\begin{align*}
J(G, \Gamma)&:=\inf\Bigg\{\liminf_{n\to+\infty} 
\left[\int_\Omega W(x,v_n(x),\nabla v_n(x)) \, \de x + 
\int_{S_{v_n}} \Psi(x, [v_n](x), \nu_{v_n}(x)) \,\de \mathcal H^{N-1}(x)\right]:  \\
&\hspace{2,5cm} v_n \in SBV(\Omega;\mathbb R^{d \times N}), v_n \to G 
\hbox{ in } L^1(\Omega;\mathbb R^{d \times N}), \nabla v_n \overset{\ast}{\rightharpoonup} \Gamma \mbox{ in } 
\mathcal M(\Omega;\mathbb R^{d\times N^2}) \Bigg\}.
\end{align*}

It was proved in \cite[Proposition 4.6]{BMMO2017} that the functional $J$ admits
the following alternative characterization, 
\begin{align}\label{I2Gfixed}
J(G, \Gamma)&:=\inf\Bigg\{\liminf_{n\to+\infty} 
\left[\int_\Omega W(x,G(x),\nabla v_n(x)) \, \de x + 
\int_{S_{v_n}} \Psi(x, [v_n](x), \nu_{v_n}(x)) \,\de \mathcal H^{N-1}(x)\right]: \nonumber \\
&\hspace{2,5cm} v_n \in SBV(\Omega;\mathbb R^{d \times N}), v_n \to G 
\hbox{ in } L^1(\Omega;\mathbb R^{d \times N}), \nabla v_n \overset{\ast}{\rightharpoonup} \Gamma \mbox{ in } 
\mathcal M(\Omega;\mathbb R^{d\times N^2}) \Bigg\},
\end{align}
in particular the density of the bulk term does not depend explicitly on the sequence $v_n$, as we can fix the second variable equal to $G$, but only on the gradient of $v_n$.

In the above expression the density $W$ satisfies the hypotheses 
\begin{enumerate}
\item[(I)]\label{(W1)_psec}  (Lipschitz continuity) there exists a constant $C_W >0$ such that, for all $x\in\Omega$, $A_1,A_2 \in \R{d\times N}$ and 
$M_1, M_2 \in \mathbb R^{d \times N^2}$,
\begin{equation*}
		|W(x,A_1,M_1) - W(x,A_2,M_2)| \leq C_W (|A_1 - A_2|+ |M_1-M_2|);
\end{equation*}
\item[(II)] \label{W4sec} there exists $c_W>0$ such that, for every $(x,A, M)\in \Omega \times \mathbb R^{d\times N} \times \mathbb R^{d \times N^2}$,
	\begin{equation*}
		\frac{1}{c_W}( |A|+ |M|)-c_W\leq W(x,A, M) \leq c_W(1+|A|+|M|);
	\end{equation*}
\item[(III)] \label{modcont} there exists a continuous function $\omega_W:[0,+\infty)\to[0,+\infty)$ such that 
$\displaystyle \lim_{t \to 0^+} \omega_W(t)=0$ and
$$|W(x_1,A, M)-W(x_0,A,M)|\leq\omega_W(|x_1-x_0|)(1+|A|+ |M|),$$
for every $x_0, x_1 \in \Omega$, $A \in \mathbb R^{d\times N}$, 
$M \in \mathbb R^{d \times N^2}$;
\item[(IV)]\label{Winftysec} 
there exists $\alpha \in (0,1)$ and $L>0$ such that 	
$$\bigg |W^{\infty}(x,A,M) - \frac{W(x,A,tM)}{t}\bigg| \leq \frac{C}{t^{\alpha}},$$ 	
for all  $t>L$, $x \in \Omega$, $A \in \mathbb R^{d \times N}$ and 
$M \in \mathbb R^{d \times N^2}$ with $|M|=1,$
where $W^{\infty}$ denotes the recession function at infinity of $W(x,A,\cdot)$; 
\end{enumerate}
and $\Psi$ satisfies, in the matrix setting, \eqref{psi_0}-\eqref{(psi4)} considered in Theorem \ref{representation}.

Under hypotheses (I), (II), \eqref{psi_0}-\eqref{(psi1)} considering the localized version of $J$, defined in 
$SD(\Omega;\mathbb R^{d \times N})\times \mathcal O(\Omega)$ by
\begin{align*}
J(G, \Gamma;O)&:=\inf\Bigg\{\liminf_{n\to+\infty} 
\left[\int_O W(x,G(x),\nabla v_n(x)) \, \de x + 
\int_{S_{v_n}\cap O} \Psi(x, [v_n](x), \nu_{v_n}(x)) \,\de \mathcal H^{N-1}(x)\right]: \nonumber \\
&\hspace{2,5cm} v_n \in SBV(O;\mathbb R^{d \times N}), v_n \to G 
\hbox{ in } L^1(O;\mathbb R^{d \times N}), \nabla v_n \overset{\ast}{\rightharpoonup} \Gamma \mbox{ in } 
\mathcal M(O;\mathbb R^{d\times N^2}) \Bigg\},
\end{align*}
it was shown in \cite[Theorem 4.5]{BMMO2017} that $J$ is the restriction to the open subsets of $\Omega$ of a Radon measure. Standard diagonalization arguments prove that it is sequentially lower semicontinuous in 
$L^1(O;\mathbb R^{d \times N})\times \mathcal M(O;\mathbb R^{d \times N^2})$, from which locality follows.

Using the characterization of $J$ given in \eqref{I2Gfixed}
it is also easy to see that the growth hypothesis (H4) from Theorem \ref{GMthmHSD} holds. 

Hence, denoting by $m$ the functional defined, in the matrix setting, by \eqref{mdef} with $L=1$, this result applies to yield the following representation:
\begin{align}\label{reprelaxsec}
J(G, \Gamma) =  \int_\Omega f(x, G(x),\nabla G(x), \Gamma(x)) \, \de x 
+ \int_{\Omega\cap S_G}\Phi(x, [G](x), \nu_G(x)) \, \de \mathcal H^{N-1}(x),
\end{align} 
where
\begin{align*}
f(x_0, A, B, D) :=	\limsup_{\varepsilon \to 0^+}
\frac{m(A + B(\cdot-x_0), D; Q(x_0,\varepsilon))}{\varepsilon^N},
\end{align*}
and
\begin{align*}
\Phi(x_0,\lambda-\theta, \nu) := \limsup_{\varepsilon \to 0^+}
\frac{m(u_{\lambda-\theta,\nu}, 0; Q_{\nu}(x_0,\varepsilon))}{\varepsilon ^{N-1}},
\end{align*}
for all $x_0\in \Omega$, $A, \lambda, \theta \in \mathbb R^{d\times N}$, 
$B, D \in \mathbb R^{d \times N^2}$, $\nu \in \mathbb S^{N-1}$, 
where $0$ denotes the zero $\mathbb R^{d \times N^2}$ matrix and 
$u_{\lambda-\theta,\nu}(y) := \begin{cases} 
\lambda - \theta, &\hbox{if } y\cdot \nu > 0\\
0, &\hbox{ if } y\cdot \nu \leq 0.\end{cases}$

The reasoning is similar to the one presented in Theorem \ref{representation},
without the translation invariance, as hypotheses (I) and (II) ensure that $W(\cdot,G(\cdot),\cdot)$ is Carathéodory in 
$\Omega \times \mathbb R^{d \times N^2}$. 

Assuming, in addition, that hypotheses (III), (IV), \eqref{(psi2)}-\eqref{(psi4)} hold, we will now show that the relaxed densities $f$ and $\Phi$ coincide with those obtained in \cite[Theorems 3.2 and 5.7]{BMMO2017}, to this end we use the alternative characterization of $J$ given in \eqref{I2Gfixed}.
Theorems 3.2 and 5.7 in \cite{BMMO2017} provide the following integral representation
for $J$
$$J(G, \Gamma) =  \int_\Omega W_2(x, G(x),\nabla G(x), \Gamma(x)) \, \de x 
+ \int_{\Omega\cap S_G}\gamma_2(x, [G](x), \nu_G(x)) \, \de \mathcal H^{N-1}(x),$$
where
\begin{align}\label{W2}
&W_{2}(x_0,A,B,D)=\hspace{0cm}\inf_{u\in SBV(Q;{\mathbb{R}}^{d\times N})}%
\hspace{0cm}\left\{ \int_{Q}W(x_0,A,\nabla u(y))\,\de y+\int_{S_{u}\cap Q}\hspace{%
0cm}\Psi(x_0,[u](y),\nu _{u}(y))\,\de\cH^{N-1}(y):\right. \nonumber\\
&\left. \hspace{7cm}u|_{\partial Q}(y)=B\cdot y,\;\int_{Q}\nabla
u(y)\,\de y=D\right\},
\end{align}
\begin{align}\label{gamma2}
&\gamma _{2}(x_0,A,\lambda-\theta ,\nu )=\hspace{0cm}\inf_{u\in SBV(Q_{\nu };{%
\mathbb{R}}^{d\times N})}\hspace{0cm}\left\{ \int_{Q_{\nu }}\hspace{0cm}%
W^{\infty }(x_0,A,\nabla u(y))\,\de y+\int_{S_{u}\cap Q_{\nu }}\hspace{0cm}\Psi
(x_0,[u](y),\nu _{u}(y))\,\de\cH^{N-1}(y):\right. \nonumber\\
&\left. \hspace{7cm}u|_{\partial Q_{\nu }}
=u_{\lambda-\theta ,\nu},\;\int_{Q_{\nu }}\nabla u(y)\,\de y=0\right\}.
\end{align}
In order to show that the densities in \eqref{reprelaxsec} are given by \eqref{W2} and \eqref{gamma2},
$$f(x_0,A,B,D)=  W_2(x_0,A,B,D) \quad \mbox{ and } \quad
\Phi(x_0,\lambda-\theta, \nu) = \gamma_2(x_0, A,\lambda -\theta, \nu),$$
for all $x_0 \in \Omega$,  $A, \lambda, \theta \in \mathbb R^{d\times N}$, 
$B, D \in \mathbb R^{d \times N^2}$, $\nu \in \mathbb S^{N-1}$, we begin by
stressing the fact that the dependence of $\gamma_2$ on $A$ is fictitious. Indeed assumptions (I) and (IV)  
guarantee that $W^\infty$ 
does not depend on $A$, i.e. 
$$W^\infty(x,A,M)= W^{\infty}(x, 0,M), \mbox{ for a.e. } x \in \Omega, 
\forall A \in \mathbb R^{d \times N}, M \in \mathbb R^{d \times N^2},$$ 
where $0$ represents the zero matrix in $\mathbb R^{d \times N}$.

As $x_0$ and $A$ are fixed, we may invoke Propositions 3.1 and 4.1 in \cite{CF1997} to conclude that
\begin{align*}
&W_{2}(x_0,A,B,D)=\hspace{0cm}\inf\hspace{0cm}\Bigg\{\liminf_{n\to\infty}
\left[ \int_{Q}W(x_0,A,\nabla u_n(y))\,\de y+\int_{S_{u_n}\cap Q}
\hspace{0cm}\Psi(x_0,[u_n](y),\nu _{u}(y))\,\de\cH^{N-1}(y)\right]: \\
&\hspace{3.5cm}u_n\in SBV(Q;{\mathbb{R}}^{d\times N}), 
u_n \to B\cdot y \mbox{ in } L^1(Q;\mathbb R^{d\times N}), 
\nabla u_n \overset{\ast}{\rightharpoonup} D \mbox{ in } 
\mathcal M(Q;\mathbb R^{d \times N^2})\Bigg\},
\end{align*}
and
\begin{align}\label{gamma2seq}
&\gamma _{2}(x_0,\lambda-\theta ,\nu) =\inf\Bigg\{\liminf_{n\to\infty} \left[\int_{Q_{\nu }}W^{\infty }(x_0,0,\nabla u_n(y))\,\de y
+\int_{S_{u_n}\cap Q_{\nu}}
\Psi(x_0,[u_n](y),\nu_{u_n}(y))\,\de\cH^{N-1}(y)\right]: \nonumber\\
&\hspace{3.5cm} u_n\in SBV(Q_{\nu};\mathbb R^{d\times N}), 
u_n \to u_{\lambda-\theta,\nu} \mbox{ in } L^1(Q_\nu;\mathbb R^{d\times N}),
\nabla u_n \overset{\ast}{\rightharpoonup} 0\mbox{ in } 
\mathcal M(Q;\mathbb R^{d \times N^2})\Bigg\},
\end{align}
respectively.

Reasoning as in the proof of Theorem \ref{representation},
by Theorem \ref{thm4.3FHP} and the fact that 
$\mathcal F= J$,  
for every $x_0\in \Omega$, $\lambda, \theta \in \mathbb R^{d}$ and
$\nu \in \mathbb S^{N-1}$, we have
\begin{align*}
\Phi(x_0,\lambda,\theta,\nu) &= \Phi(x_0,\lambda -\theta,\nu) = 
\limsup_{\varepsilon \to 0^+}\frac{m(u_{\lambda-\theta,\nu}(\cdot-x_0), 0; Q_{\nu}(x_0,\varepsilon))}{\varepsilon ^{N-1}} \\
&=\limsup_{\varepsilon \to 0^+}\frac{J(u_{\lambda-\theta,\nu}(\cdot-x_0), 0; Q_{\nu}(x_0,\varepsilon))}{\varepsilon ^{N-1}}\\
&= \limsup_{\varepsilon \to 0^+}\frac{1}{\varepsilon^{N-1}}
\inf\Bigg\{\liminf_{n\to+\infty}\Bigg[\int_{Q_\nu(x_0,\varepsilon)} 
W(x,u_{\lambda - \theta,\nu}(x-x_0),\nabla v_n(x)) \, \de x \\
&\hspace{6,7cm} + \int_{Q_\nu(x_0,\varepsilon)\cap S_{v_n}}
\Psi(x, [v_n](x),\nu_{v_n}(x)) \,\de\mathcal H^{N-1}(x)\Bigg] :\\
\\
&\hspace{3,7cm} v_n \in SBV(Q_\nu(x_0,\varepsilon);\mathbb R^{d\times N}), 
\nabla v_n \overset{\ast}{\rightharpoonup} 0 
\mbox{ in } \mathcal M(Q;\mathbb R^{d \times N^2}),\\
&\hspace{3,7cm}  v_n \to u_{\lambda-\theta,\nu}(\cdot-x_0) \mbox{ in }  L^1(Q_\nu(x_0,\varepsilon);\mathbb R^{d\times N})\Bigg\} \\
&= \limsup_{\varepsilon \to 0^+}
\inf\Bigg\{\liminf_{n\to+\infty} \Bigg[\int_{Q_\nu} \varepsilon 
W\left(x_0+\varepsilon y,u_{\lambda - \theta,\nu}(y),
\frac{1}{\varepsilon}\nabla u_n(y)\right) \, \de y \\
&\hspace{6,7cm} +\int_{Q_\nu\cap S_{u_n}}
\Psi(x_0+\varepsilon y, [u_n](y),\nu_{u_n}(y)) \,\de\mathcal H^{N-1}(y) \Bigg]:\\
&\hspace{2,7cm} u_n \in SBV(Q_\nu;\mathbb R^d), 
u_n \to u_{\lambda-\theta,\nu} \mbox{ in }  L^1(Q_\nu;\mathbb R^d),
\nabla u_n \overset{\ast}{\rightharpoonup} 0 
\mbox{ in } \mathcal M(Q;\mathbb R^{d \times N^2}) \Bigg\},
\end{align*}
where in the last equality we performed a change of variables. Using hypotheses (III)
and \eqref{(psi4)} first, and then (IV), we obtain

\vfill

\begin{align*}
\Phi(x_0,\lambda -\theta,\nu) &= \limsup_{\varepsilon \to 0^+}
\inf\Bigg\{\liminf_{n\to+\infty} \Bigg[\int_{Q_\nu} \varepsilon 
W\left(x_0,u_{\lambda - \theta,\nu}(y),
\frac{1}{\varepsilon}\nabla u_n(y)\right) \, \de y \\
&\hspace{6,7cm} +\int_{Q_\nu\cap S_{u_n}}
\Psi(x_0, [u_n](y),\nu_{u_n}(y)) \,\de\mathcal H^{N-1}(y) \Bigg]:\\
&\hspace{2,3cm} u_n \in SBV(Q_\nu;\mathbb R^d), 
u_n \to u_{\lambda-\theta,\nu} \mbox{ in }  L^1(Q_\nu;\mathbb R^d),
\nabla u_n \overset{\ast}{\rightharpoonup} 0 
\mbox{ in } \mathcal M(Q;\mathbb R^{d \times N^2}) \Bigg\}\\
&= \inf\Bigg\{\liminf_{n\to+\infty} \Bigg[\int_{Q_\nu}  
W^{\infty}\left(x_0,u_{\lambda - \theta,\nu}(y),\nabla u_n(y)\right) \, \de y 
+\int_{Q_\nu\cap S_{u_n}}\hspace{-0,75cm}
\Psi(x_0, [u_n](y),\nu_{u_n}(y)) \,\de\mathcal H^{N-1}(y) \Bigg]:\\
&\hspace{2,3cm} u_n \in SBV(Q_\nu;\mathbb R^d), 
u_n \to u_{\lambda-\theta,\nu} \mbox{ in }  L^1(Q_\nu;\mathbb R^d),
\nabla u_n \overset{\ast}{\rightharpoonup} 0 
\mbox{ in } \mathcal M(Q;\mathbb R^{d \times N^2}) \Bigg\}
\end{align*}
\begin{align*}
&= \inf\Bigg\{\liminf_{n\to+\infty} \Bigg[\int_{Q_\nu}  
W^{\infty}\left(x_0,0,\nabla u_n(y)\right) \, \de y 
+\int_{Q_\nu\cap S_{u_n}}\hspace{-0,5cm}
\Psi(x_0, [u_n](y),\nu_{u_n}(y)) \,\de\mathcal H^{N-1}(y) \Bigg]:\\
&\hspace{2,3cm} u_n \in SBV(Q_\nu;\mathbb R^d), 
u_n \to u_{\lambda-\theta,\nu} \mbox{ in }  L^1(Q_\nu;\mathbb R^d),
\nabla u_n \overset{\ast}{\rightharpoonup} 0 
\mbox{ in } \mathcal M(Q;\mathbb R^{d \times N^2}) \Bigg\}\\
&= \gamma_2(x_0, \lambda - \theta, \nu).
\end{align*}

We have thus proved the equality between the densities of the surface term
$$\Phi(x_0,\lambda,\theta,\nu) =  \gamma_2(x_0, \lambda - \theta, \nu),$$
for every $x_0\in \Omega$, $\lambda, \theta \in \mathbb R^{d \times N}$ and 
$\nu \in \mathbb S^{N-1}$, where $\gamma_2$ is given in \eqref{gamma2seq}.
 
The proof that $f(x_0,A,B,D)=  W_2(x_0,A,B,D)$ is similar and we omit the details.
 
Thus our integral representation for $J$, obtained via the global method given in Theorem \ref{GMthmHSD}, recovers the one proved in \cite[Theorems 3.2 and 5.7]{BMMO2017}, and we emphasize the fact that $\gamma_2$ given in \eqref{gamma2} does not really depend on the variable $A$.

\subsection{Multi-levelled structured deformations}\label{MSD}

The global method for relaxation that we propose in Theorem \ref{GMthmHSD}, allows us not only to recover the recursive relaxation procedure presented in \cite[Subsection 3.2 and Theorem 3.4]{BMMOZ2022}, since in each step the obtained densities satisfy hypotheses \eqref{(W1)_p}-(9), thus entitling us to apply Theorems \ref{representation} and \ref{contrepresentation}, but also to propose an alternative direct strategy.

Indeed, as it will be rigorously shown in \cite{BMMOZ2024}, one can associate to each multi-levelled structured deformation an energy satisfying hypotheses (H1)-(H4)
in Section \ref{gm}. The alternative procedure of assigning this energy 
very possibly yields an expression lower than the one obtained in \cite[Subsection 3.2]{BMMOZ2022}.

More precisely, referring to the case $L=2$ for simplicity of exposition, starting from \eqref{103}
and given $(g,G_1,G_2,G_3)\in HSD^p_2(\Omega)$, Theorem \ref{GMthmHSD} provides an integral representation for the functional
 \begin{equation*}
 	\begin{split}
 		I_3(g,G_1, G_2,G_3)\coloneqq \inf_{\{u_{n_1, n_2,n_3}\}\subset SBV(\Omega;\R{d})}\Big\{  \liminf_{n_1\to +\infty,n_2\to+\infty, n_3 \to +\infty} E(u_{n_1, n_2,n_3}): u_{n_1, n_2,n_3}\Hto(g,G_1,G_2,G_3), \\ 
 	\|\nabla g_{n_1}\|_{L^p(\Omega;\mathbb R^{d\times N})}, \|\nabla g_{n_1n_2}\|_{L^p(\Omega;\mathbb R^{d\times N})}, \|\nabla u_{n_1,n_2,n_3}\|_{L^p(\Omega;\R{d\times N})}
 	<+\infty\Big\},
 \end{split}
 \end{equation*}
 where $\displaystyle g_{n_1}:=\lim_{n_2,n_3}u_{n_1n_2n_3}$, $\displaystyle g_{n_1n_2}:=\lim_{n_3}u_{n_1n_2n_3}$, the convergences are those in Definition \ref{S000} and $I_3$ represents, by definition, the energy assigned to the three-levelled structured deformation $(g,G_1,G_2,G_3)$.

\bigskip

\subsection*{Acknowledgements}

The authors would like to thank the referee for his/her valuable comments which lead to the improvement of the manuscript.

The research of ACB was partially supported by National Funding from FCT -
Funda\c c\~ao para a Ci\^encia e a Tecnologia through project 
UIDB/04561/2020:  https://doi.org/10.54499/UIDB/04561/2020 
and also by GNAMPA, Programma Professori Visitatori, year 2023.

The research of JM was supported by GNAMPA, Programma Professori Visitatori, year 2022  and through FCT/Portugal through CAMGSD, IST-ID,
projects UIDB/04459/2020 and UIDP/04459/2020.
He also gratefully acknowledges the support and hospitality of Sapienza-University of Rome through the Programma
Professori Visitatori, year 2023.

EZ acknowledges the support of Piano Nazionale di Ripresa e Resilienza (PNRR) - Missione 4 ``Istruzione e Ricerca''
- Componente C2 Investimento 1.1, "Fondo per il Programma Nazionale di Ricerca e
Progetti di Rilevante Interesse Nazionale (PRIN)" - Decreto Direttoriale n. 104 del 2 febbraio 2022 - CUP 853D23009360006. She is a member of the Gruppo Nazionale per l'Analisi Matematica, la Probabilit\`a e le loro Applicazioni (GNAMPA) of the Istituto Nazionale di Alta Matematica ``F.~Severi'' (INdAM). 
She also acknowledges partial funding from the GNAMPA Project 2023 \emph{Prospettive nelle scienze dei materiali: modelli variazionali, analisi
asintotica e omogeneizzazione}.
The work of EZ is also supported by Sapienza - University of Rome through the projects Progetti di ricerca medi, (2021), coordinator  S. Carillo e Progetti di ricerca piccoli,  (2022), coordinator E. Zappale.
She gratefully acknowledges the hospitality and support of CAMGSD, IST-ID.

\bibliographystyle{plain}

\end{document}